\documentclass{amsart}
\usepackage{amsmath}
\usepackage{amsthm}
\usepackage{mathrsfs}
\usepackage{amsfonts}
\usepackage{graphicx}
\usepackage{color}
\usepackage{verbatim}

\newtheorem{theorem}{Theorem}[section]
\newtheorem{proposition}[theorem]{Proposition}
\newtheorem{lemma}[theorem]{Lemma}
\newtheorem{corollary}[theorem]{Corollary}

\theoremstyle{definition}
\newtheorem{definition}[theorem]{Definition}

\theoremstyle{remark}

\newtheorem{remark}[theorem]{Remark}
\newtheorem{example}[theorem]{Example}

\author{Christopher Severs}
\address{Reykjav\'{i}k University, Menntavegur 1, IS 101 Reykjav\'{i}k, Iceland}
\email{csevers@ru.is}
\thanks{
The first author was partially supported by an NSF grant DMS-0441170,
administered by the Mathematical Sciences Research Institute while the
author was in residence at MSRI during the Complementary program, Fall 2009 - Spring 2010.
This work was developed during the
visit of the author to MSRI and we thank the institute for its
hospitality.
}

\author{Jacob A. White}
\address{Mathematical Sciences Research Institute, Berkeley, CA}
\email{jawhite@msri.org}
\thanks{
The second author was partially supported by an NSF grant DMS-0932078,
administered by the Mathematical Sciences Research Institute while the
author was in residence at MSRI during the Complementary Program, Fall 2010.
This work was finished  during the
visit of the author to MSRI and we thank the institute for its
hospitality.
}


\date{}


\title{On the Homology of the Real Complement of the $k$-Parabolic Subspace Arrangement}

\begin{document}

\begin{abstract}
In this paper, we study $k$-parabolic arrangements, a generalization of the $k$-equal arrangement for any finite real reflection group. 
When $k=2$, these arrangements correspond to the well-studied Coxeter arrangements. 
We construct a cell complex $Perm_k(W)$ that is homotopy equivalent to the complement. We then apply discrete Morse theory to obtain a minimal cell complex for the complement. 
As a result, we give combinatorial interpretations for the Betti numbers, and show that the homology groups are torsion free. We also study a generalization of the Independence 
Complex of a graph, and show that this generalization is shellable when the graph is a forest. This result is used in studying $Perm_k(W)$ using discrete Morse theory.

\end{abstract}

\maketitle

\section{Introduction}

A \emph{subspace arrangement} $\mathscr{A}$ is a collection of linear subspaces of a finite-dimensional vector space $V$, 
such that there are no proper containments among the subspaces. 
Examples of subspace arrangements include real and complex \emph{hyperplane arrangements}. 
One of the main questions regarding subspace arrangements is to study the structure of the 
\emph{complement} $\mathcal{M}(\mathscr{A}) = V - \cup_{X \in \mathscr{A}} X$. 
Many results regarding the homology and homotopy theory of $\mathcal{M}(\mathscr{A})$ can be found in the book \emph{Arrangements of Hyperplanes} by Orlik and Terao \cite{orlik-terao},
when $\mathscr{A}$ is a real or complex \emph{hyperplane} arrangement.

The base example that serves as motivation for this paper is the $k$-equal arrangement over $\mathbb{R}$. The $k$-equal arrangement, $\mathscr{A}_{n-1, k}$ is the collection of 
subspaces given by equations:
$$x_{i_1} = \ldots = x_{i_k}$$
over all distinct $1 \leq i_1 < i_2 < \ldots < i_k \leq n$. Note that this is a subspace arrangement over $\mathbb{R}^n$.
This subspace arrangement was originally investigated in connection with the $k$-equal problem: given $n$ real numbers, determine 
whether or not some $k$ of them are equal \cite{bjorner-lovasz}.

Given a topological space $X$, let $\widetilde{\beta}_i(X)$ denote the ranks of the torsion-free part of the $i$th singular reduced integral homology group. Given a real subspace arrangement $\mathscr{A}$, Bj\"orner and Lov\'asz showed that the minimum number of leaves in a linear decision tree 
that determines membership in $\mathscr{A}$ is at least $1 + 2 \sum_{i \geq 0} \widetilde{\beta}_i(\mathcal{M}(\mathscr{A}))$. Thus, knowing the homology of the complement gives lower bounds on the minimum depth of a linear decision tree which decides membership in $\mathscr{A}$.

A combinatorial tool that has proven useful in studying the complement is the intersection lattice, $\mathcal{L}(\mathscr{A})$, 
which is the lattice of intersections of subspaces, ordered by reverse inclusion. In particular, the work of Goresky and MacPherson gives an isomorphism, 
known as the Goresky MacPherson formula, that allows one to translate the problem of 
determining the homology groups of $\mathcal{M}(\mathscr{A})$ into the problem of studying certain groups related to $\mathcal{L}(\mathscr{A})$. 

\begin{theorem}[Theorem III.1.3 in Stratified Morse Theory \cite{goresky-macpherson}]
Let $\mathscr{A}$ be a real linear subspace arrangement.
We have the following isomorphism:
$$H^i(\mathcal{M}(\mathscr{A})) \cong \bigoplus_{x \in \mathcal{L}(\mathscr{A})_{> \hat{0}}} H_{n- \dim(x) -i - 2}(\Delta(\hat{0}, x))$$
where $\Delta(\hat{0}, x)$ is the order complex of the interval $[\hat{0}, x]$, and dimension refers to dimension over $\mathbb{R}$.

\end{theorem}

In \cite{bjorner-welker}, Bj\"orner and Welker used the Goresky MacPherson formula to determine the cohomology groups of the complement of the $k$-equal arrangement. They showed that the groups are torsion free, and trivial in dimensions that are not a multiple of $k-2$. They also obtained formulas for the homology groups of the order complex of $\mathcal{L}(\mathscr{A}_{n,k})$. 
Similar formulas were obtained by Bj\"orner and Wachs \cite{bjorner-wachs} by showing $\mathcal{L}(\mathscr{A}_{n,k})$ has an $EL$-labeling, and thus is shellable. We note that Bj\"orner and Wachs extended the definition of shellability to non-pure posets in order to obtain their results. 

Type $B$ and $D$ analogues of the $k$-equal arrangement were studied by Bj\"orner and Sagan \cite{subspacesBD}. Their type $B$ analogue is denoted $\mathscr{B}_{n,k,h}$, and their type $D$ analogue is denoted $\mathscr{D}_{n,k}$. The arrangement $\mathscr{D}_{n,k}$ consists of subspaces given by equations:

$$\epsilon_1 x_{i_1} = \cdots = \epsilon_k x_{i_k}$$

over all $\{i_1, \ldots, i_k \} \subset [n]$ and all $(\epsilon_1, \ldots, \epsilon_k) \in \{+,-\}^k$. The arrangement $\mathscr{B}_{n,k,h}$ consists of $\mathscr{D}_{n,k}$ and new subspaces given by equations:

$$x_{i_1} = \cdots = x_{i_h} = 0$$

over all $\{i_1, \ldots, i_h\} \subset [n]$. 

Bj\"orner and Sagan showed that the 
intersection lattice $\mathcal{L}(\mathscr{B}_{n,k,h})$ has an $EL$-labeling, and obtained results regarding the cohomology of $\mathcal{M}(\mathscr{B}_{n,k,h})$. Their methods did not extend to $\mathscr{D}_{n,k}$. Kozlov and Feichtner \cite{subspacesD} obtained results regarding the cohomology of the complement of $\mathscr{D}_{n,k}$, by showing $\mathcal{L}(\mathscr{D}_{n,k})$ had an $EC$-labeling, a notion due to Kozlov \cite{kozlov-orbit}.

We see that using the Goresky MacPherson formula in general is a challenge. It translates the problem of studying the cohomology groups of the complement into a problem of finding the homology groups of the order complex
of the intersection lattice. Most of these examples involved expanding previously known methods in order to find equations for the Betti numbers. Also, in most cases, the problem of writing down an explicit 
general formula for the Betti numbers of the order complex is very difficult. In the end, one does not get much intuition regarding what these Betti numbers are actually counting.
Finally, the labelings involved for the type $D$ $k$-equal arrangement are very different from the labelings used for $\mathscr{B}_{n,k,h}$, so one might be skeptical about giving a uniform proof for a generalization of these 
arrangements given for arbitrary reflection group.

In a previous paper \cite{full-version}, we introduced a generalization of the $k$-equal arrangement associated to any finite Coxeter group $W$. We denote this arrangement, called the $k$-parabolic arrangement, by $\mathscr{W}_{n,k}$. These arrangements correspond to orbits of subspaces fixed by irreducible parabolic subgroups of rank $k-1$. In \cite{full-version}, we studied the fundamental group of the complement of these arrangements, generalizing work that had been done by Khovanov \cite{khovanov} for the $3$-equal arrangements of types $A$, $B$, and $D$. In this paper, we study the integral homology groups of the complement. 

Given the challenges of proving shellability of the intersection lattice for the previously studied cases, we choose to study the homology groups using a different approach. Fix a finite Coxeter group $W$ of rank $n$, and let $3 \leq k \leq n$. First, we note that the $k$-parabolic arrangement is always embedded in the corresponding reflection arrangement of $W$. Using a construction due to Solomon, we obtain a cell complex $Perm_k(W)$, that is homotopy equivalent to $\mathcal{M}(\mathscr{W}_{n,k})$.
Then we use Forman's discrete Morse theory \cite{formanmorse} to study $Perm_k(W)$. We note that discrete Morse theory has been applied multiple times in recent years in topological combinatorics. It has also been applied to study complements of complex hyperplane arrangements \cite{salvetti-settepanella}. This marks the first attempt to use discrete Morse theory to study the complement of a subspace arrangement. In particular, we obtain a minimal cell complex for $Perm_k(W)$: that is, a homotopy equivalent cell complex with exactly $\beta_i(Perm_k(W))$ cells of dimension $i$.

\begin{theorem}
\label{thm:minimality}
 There exists a minimal cell complex $M_k(W)$ such that $Perm_k(W) \cong M_k(W)$.
\end{theorem}
We note that when $k = 3$, $\mathcal{M}(\mathscr{W}_{n,3})$ is not simply connected. For simply-connected topological spaces $X$, there is already a well-known construction \cite{hatcher} of a minimal cell complex that is homotopy equivalent to $X$. To our knowledge, Theorem \ref{thm:minimality} is the first result regarding existence of minimal cell complexes for real subspace arrangements whose complements 
have nontrivial fundamental groups.

We hope that $M_k(W)$ can be used to study the cohomology ring structure of $\mathcal{M}(\mathscr{W}_{n,k})$, something that cannot be done using the Goresky-MacPherson formula. However, for this paper, however, we focus on obtaining information about the homology groups from $M_k(W)$.
Let $H_{i,k}(W)$ be the 
$i$th singular homology group of $\mathcal{M}(\mathscr{W}_{n,k})$, and let $\beta_{i,k}(W)$ be the rank of the torsion-free part of 
$H_{i,k}(W)$. Then we obtain the following results.

\begin{theorem}
 \label{thm:homology}
Let $W$ be a Coxeter group of rank $n$, and let $3 \leq k \leq n$. Then the following holds:
\begin{enumerate}
 \item $H_{i,k}(W)$ is torsion-free.
\item $H_{i,k}(W)$ is trivial unless $i = t(k-2)$ for some $0 \leq t \leq \frac{n}{k}$.
\end{enumerate}

\end{theorem}
We note that these results were obtained when $W$ is of type $A$, $B$ or $D$ by Bj\"orner and Welker \cite{bjorner-welker}, 
Bj\"orner and Sagan \cite{subspacesBD}, and Kozlov and Feichter \cite{subspacesD}, respectively.

\begin{figure}[htbp]

 \begin{center}
  \includegraphics[height=10cm]{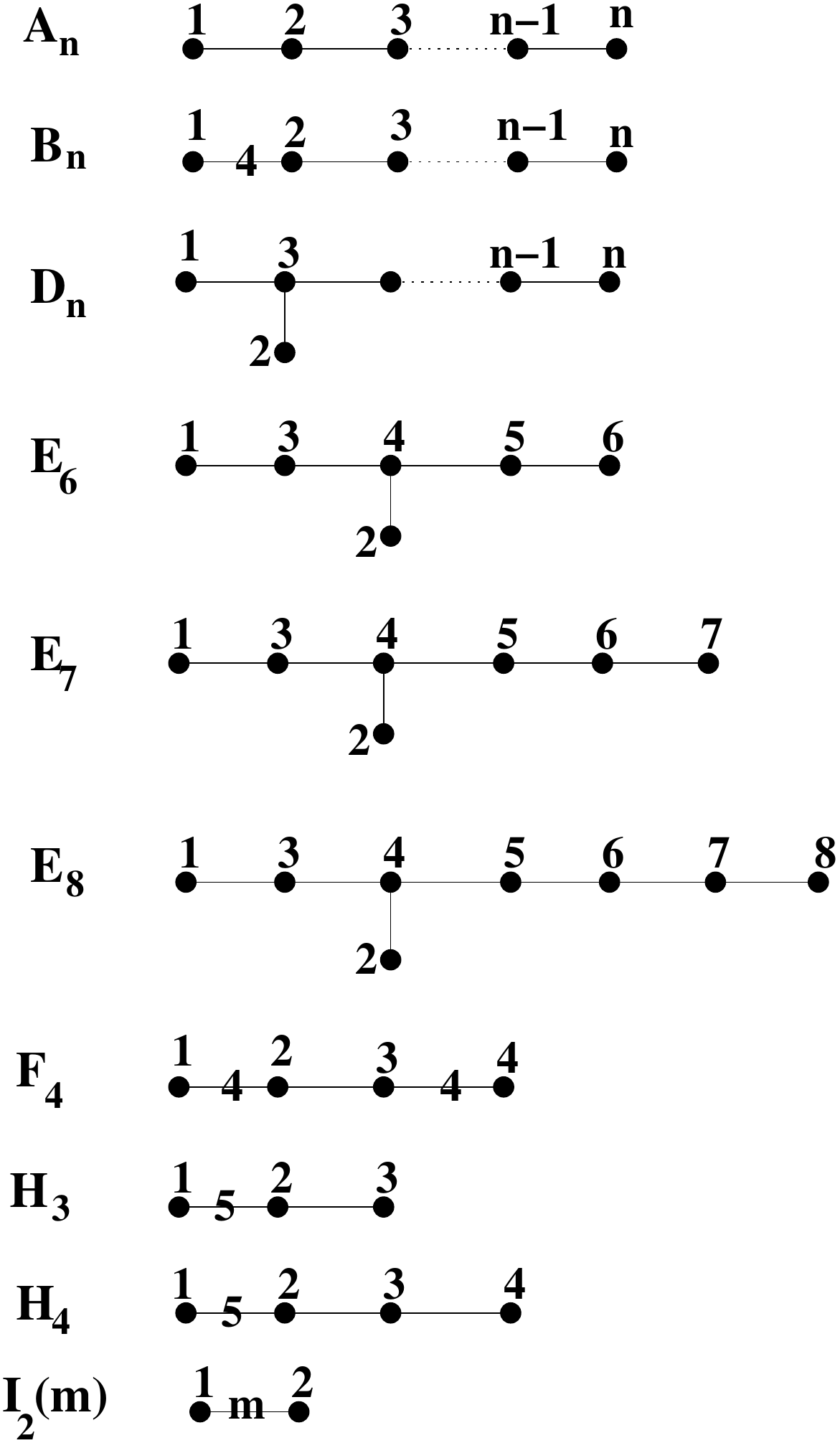}
 \caption{linear orders for irreducible Coxeter groups}
 \end{center}
\label{fig:linearorder}
\end{figure}

We also have a combinatorial interpretation of the Betti numbers.
For now, let $W$ be irreducible, with set of simple reflections $S$. Let $D$ be the Dynkin diagram for $W$. Furthermore, suppose $S$ is linearly ordered according to the numbering of verices appearing in Figure \ref{fig:linearorder}. Finally, given a set $T \subset S$ let $\mathscr{C}$ be the vertex sets of connected components of $D[T]$, the subgraph of the Dynkin diagram induced by T. Given a component $C \in \mathscr{C}$, let $N^<_{\mathscr{C}}(C)$ be the set of vertices of $D$ that are not in any component, are adjacent to some vertex of $C$, occur in the linear order on $S$ before any of the vertices of $C$, and are not adjacent to any vertex of any other component of $D[T]$. Finally, given $w \in W$, let $Des(w)$ be the 
descent set of $w$. Then we obtain the following:
\begin{theorem}
\label{thm:bettis}
 Let $W$ be an irreducible Coxeter group of rank n, and let $3 \leq k \leq n$. Let $S$ be ordered as in Figure \ref{fig:linearorder}, and let 
$0 \leq t \leq \frac{n}{k}$ be an integer. Then $\beta_{t(k-2), k}(W)$ is the number of pairs $(w, T)$ such that:
\begin{enumerate}
\item $w \in W$, $T \subset Des(w)$.
 \item $D[T]$ has $t$ components $\mathscr{C} = \{C_1, \ldots, C_k \}$, each of size $k-2$
\item For every component $C$ of $D[T]$, we have $N^<_{\mathscr{C}}(C) \cap Des(w) \neq \emptyset$
\item For every $v \in Des(w)$, $v$ is adjacent to some component of $D[T]$.
\end{enumerate}

\end{theorem}
We note that for classical reflection groups, this interpretation can be described more explicitly. Moreover, we obtain new formulas for the Betti numbers corresponding to types 
$A$, $B$, and $D$. Finally, Theorem \ref{thm:bettis} also holds for any finite Coxeter group, and for a much larger class of linear orders than just the linear orders mentioned in 
Figure \ref{fig:linearorder}. We mention more regarding these facts in Section \ref{sec:matchalg}.

Fix a finite Coxeter group $W$, let $n$ be the rank of the Coxeter group, $S$ be the set of simple reflections, and $D$ be the Dynkin diagram. Finally, unless otherwise noted, 
$k$ is a fixed integer with $3 \leq k \leq n$.

\section{Definition of $k$-Parabolic Arrangement}
\label{sec:defkpar}

Since we are generalizing the $k$-equal arrangement, which corresponds to the case $W = A_n$, 
we use it as our motivation. In this paper, we actually work with the essentialized $k$-equal arrangement. 
Recall that the $k$-equal arrangement, $\mathscr{A}_{n,k}$, is the collection of all subspaces given by 
$x_{i_1} = x_{i_2} = \ldots = x_{i_k}$ over all indices $\{i_1, \ldots, i_k\} \subset [n+1]$, 
with the additional relation $\sum_1^{n+1}x_i = 0$.  The intersection poset 
$\mathcal{L}(\mathscr{A}_{n,k})$ is a subposet of $\mathcal{L}(\mathscr{H}(A_n))$. 
There is already a well-known combinatorial description of both of these posets. 
The poset of all set partitions of $[n+1]$ ordered by refinement is isomorphic to $\mathcal{L}(\mathscr{H}(A_n))$, 
and under this isomorphism, $\mathcal{L}(\mathscr{A}_{n,k})$ is the subposet of set partitions 
where each block is either a singleton, or has size at least $k$. 
However, our generalization relies on 
the Galois correspondence of Barcelo and Ihrig \cite{barcelo-ihrig}.

Given $I \subset S$, let $W_I$ be the subgroup generated by the reflections of 
$I$. Such a subgroup is called a \emph{standard parabolic subgroup}. A standard parabolic subgroup is \emph{irreducible} if $(W_I, I)$ is an irreducible Coxeter system. Given the Dynkin diagram, 
a subset $I \subset S$ corresponds to an irreducible standard parabolic subgroup if and only if the subgraph induced by $I$ is a connected graph.
 Any conjugate of a standard parabolic subgroup is called a \emph{parabolic subgroup}. 
We say a given subgroup is a $k$-parabolic subgroup if it is the conjugate of an \emph{irreducible} parabolic subgroup of rank $k-1$.  
Given a parabolic subgroup $G$, let $Fix(G) = \{x \in \mathbb{R}^n: wx = x, \mbox{ for all } w \in W \}$. Given a subspace $X$, let $Gal(X) = \{w \in W: wx = x \mbox{ for all } x \in X \}$. 
Let $\mathscr{P}(W)$ be the collection of all parabolic subgroups, ordered by inclusion.
Barcelo and Ihrig proved the following:

\begin{theorem}[Theorem 3.1 in \cite{barcelo-ihrig}]
\label{thm:Galois}
 The maps $G \to Fix(G)$ and $X \to Gal(X)$ defined above are lattice isomorphisms between 
$\mathscr{P}(W)$ and $\mathcal{L}(\mathscr{H}(W))$.
\end{theorem}

\begin{figure}[htbp]

 \includegraphics[height=4cm]{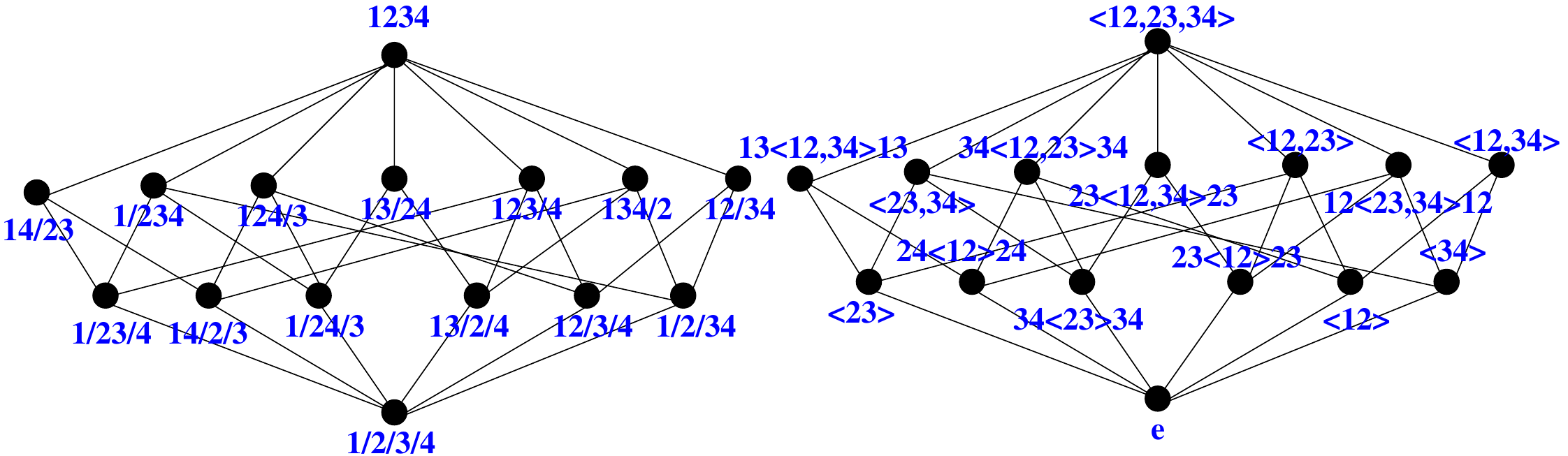}

\caption{The Galois Correspondence for $A_3$}
\label{fig:partlattice}
\end{figure}

\begin{definition}
Let $W$ be a finite real reflection group of rank $n$, and let $2 \leq k \leq n$.
Let  $\mathscr{P}_{n,k}(W)$ be the collection of \emph{all} irreducible parabolic subgroups of $W$ of rank $k-1$.

Then the $k$-parabolic arrangement $\mathscr{W}_{n,k}$ is the collection of subspaces 
$$\{ Fix(G): G \in \mathscr{P}_{n,k}(W) \}.$$
\end{definition}

Note that, given $W$, $2 \leq j \leq k \leq n$, $\mathscr{W}_{n,k}$ is embedded in $\mathscr{W}_{n,j}$. Moreover, the arrangement 
is invariant under the action of $W$. Note that in this paper, we will usually assume $k \neq 2$, as in this case we obtain the Coxeter arrangement. 
The complement of the Coxeter arrangement consists of $|W|$ disjoint regions, with no nontrivial homology above dimension $0$.
Some examples of $\mathscr{W}_{n,k}$ are given in Table \ref{tbl:kparabolic}.

\begin{table}[htbp]
\begin{center}
\begin{tabular}{|c|c|c|c|c|c|}
\hline
$W$ & $A_n$ & $B_n$ & $D_n (k = 3)$ & $D_n (k > 3$) & $W (k = 2)$ \\
\hline
$\mathscr{W}_{n,k}$ & $\mathscr{A}_{n,k}$ & $\mathscr{B}_{n,k,k-1}$ & $\mathscr{D}_{n,3}$ & $\mathscr{B}_{n,k,k-1}$ & $\mathscr{H}(W)$ \\
\hline
\end{tabular}
 \end{center}
\caption{$k$-parabolic arrangements for particular choices of $W$}
\label{tbl:kparabolic}
\end{table}

Since our motivation comes from the group action, in this paper $\mathscr{D}_{n,k}$ will always refer to the $k$-parabolic arrangement of type $D$, 
even when this arrangement is different from the previously defined analogue of the type $D$ $k$-equal arrangement.

Now we construct $Perm_k(W)$. The construction relies on the fact that $\mathscr{W}_{n,k}$ is embedded in the Coxeter arrangement, $\mathscr{H}(W)$ of type $W$. Let $\Delta(W)$ be 
the simplicial decomposition of $\mathbb{S}^{n-1}$ induced by $\mathscr{H}(W) \cap \mathbb{S}^{n-1}$. This complex is known as the Coxeter complex.
An example of a Coxeter complex is given in Figure \ref{permuteahedron}. 

Now we describe the face poset of the Coxeter complex. This description is found in Section 1.14 of Humphreys \cite{coxbook}. Given $I \subseteq S$, let 
$C_I = \{x \in \mathbb{R}^n: (x, \alpha) = 0 \mbox{ for all } \alpha \in I, (x, \beta) > 0 \mbox{ for all } \beta \in S \setminus I \}$.
Clearly this object is a convex cone. Given $I \subseteq S, w \in W$, let $wC_I = \{wx: x \in C_I\}$. These regions, when intersected with the $(n-1)$-sphere,
 correspond to faces in the Coxeter 
complex. Thus, the face poset of the Coxeter complex corresponds to cosets of standard parabolic subgroups, ordered by \emph{reverse} inclusion. We shall call such cosets \emph{parabolic cosets}. 

Let $\Delta_k(W) = \{F \in \Delta: \exists X \in \mathscr{W}_{n,k} \mbox{ such that } F \subseteq X \}$. Clearly $\Delta_k(W)$ is a subcomplex of $\Delta(W)$. 
Moreover, we have the following proposition. 

\begin{proposition}
 \label{thm:subtocom}
 $\mathcal{M}(\mathscr{W}_{n,k})$ is homotopy equivalent to $|\Delta(W)| \setminus |\Delta_k(W)|$.
\end{proposition}

\begin{proof}
Since $\mathscr{W}_{n,k}$ is essential, we are removing subspaces containing the origin. Then the map $f$ sending $x \to \frac{x}{|x|}$ on $\mathbb{R}^n$, gives 
a homotopy equivalence between $\mathcal{M}(\mathscr{W}_{n,k})$ and $|\Delta(W)| \setminus |\Delta_k(W)|$. \end{proof}

\begin{lemma}
 Let $W$ be a finite reflection group of rank $n$, let $2 \leq k \leq n$.
Then 
$\Delta_k(W)$ corresponds to cosets $wW_I$ where there exists $J \subset I$ such that $W_J$ is a $k$-parabolic subgroup.
\end{lemma}
In other words, the maximal simplices of the complex $\Delta_k(W)$ correspond to cosets $uW_I$, where the the subgraph of $D$ (the Dynkin diagram) induced by $I$ is connected, and has 
$k-1$ vertices. 

\begin{proof}
Given a coset $wW_I$, and $X$ in $\mathcal{W}_{n,k}$, we claim that 
$wC_I \subset X$ if and only if $C_I \subset w^{-1}X$. Let $x \in C_I$, $X = Fix(uW_Ju^{-1})$ where $u \in W, J \subset S$ and $uW_Ju^{-1} \in \mathscr{P}_{n,k}(W)$.
Then $wx \in Fix(uW_Ju^{-1})$ if and only if $x \in w^{-1}Fix(uW_Ju^{-1})$, so $wC_I \subset X$ if and only if $C_I \subset w^{-1}X$.

Fix $I \subset S$. Thus it suffices to understand when $C_I \subset X$ for some $X \in \mathscr{W}{n,k}$. Clearly if there exists $J \subset I$ such that $W_J$ is a $k$-parabolic 
subgroup, then $C_I \subset Fix(W_J) \in \mathscr{W}{n,k}$. So we see that the subcomplex $\Delta_k(W)$ contains all cosets $wW_I$ where there exists a standard 
$k$-parabolic subgroup $W_J$ with $J \subset I$. Now we must show that these are the only faces in $\Delta_k(W)$. 
Consider a 
standard $k$-parabolic subgroup $W_J$, and assume that $J$ is not a subset of $I$. Let $s \in J \setminus I$, and let $\alpha$ be the simple root 
corresponding to $s$. By definition of $C_I$, for any $x \in C_I$ we have $(x, \alpha) > 0$, which means that $x \not\in H_{\alpha}$, and thus $x \not\in Fix(W_J)$. Hence 
$C_I$ is disjoint from $Fix(W_J)$. Thus if $wW_I$ happens to be such that for all $J \subset I$, $W_J \not\in \mathscr{P}_{n,k}(W)$, then $wC_I$ is not in $\Delta_k(W)$. 
Hence we obtain the description of $\Delta_k(W)$ given above. \end{proof}

Next we consider a polytope related to the Coxeter complex, known in the literature as the Coxeter cell or Coxeter permutahedron $Perm(W)$. We construct $Perm_k(W)$ as a subcomplex of $Perm(W)$.
Consider a point $x$ in one of the regions of $\mathscr{H}(W)$, and let $W(x) = \{wx : w \in W \}$. For any set $I \subset W$, let $I(x) = \{wx: w \in I \}$. The $W$-permutahedron
is the convex hull of $W(x)$. An example of the $B_3$-permutahedron is given in Figure \ref{permuteahedron}. The $W$-permutahedron, denoted $Perm(W)$, is a polytope. 
It is a combinatorial exercise to show that the face poset of the $W$-permutahedron is dual of the 
face poset of the Coxeter complex. That is, there is a bijection $\varphi: \mathscr{F}(\Delta(W)) \to \mathscr{F}(Perm(W))$ such that $F \subset G$ if and only if $\varphi(G) \subset \varphi(F)$, 
for any $F, G$ in $\mathscr{F}(\Delta(W))$. 
So faces of the $W$-permutahedron correspond to parabolic cosets, ordered by inclusion.
Note that the one skeleton of the $W$ permutahedron is the Cayley graph of $W$ with respect to the generating set $S$.

\begin{figure}[htbp]
\begin{center}
\begin{tabular}{cc}
\includegraphics[height=5cm]{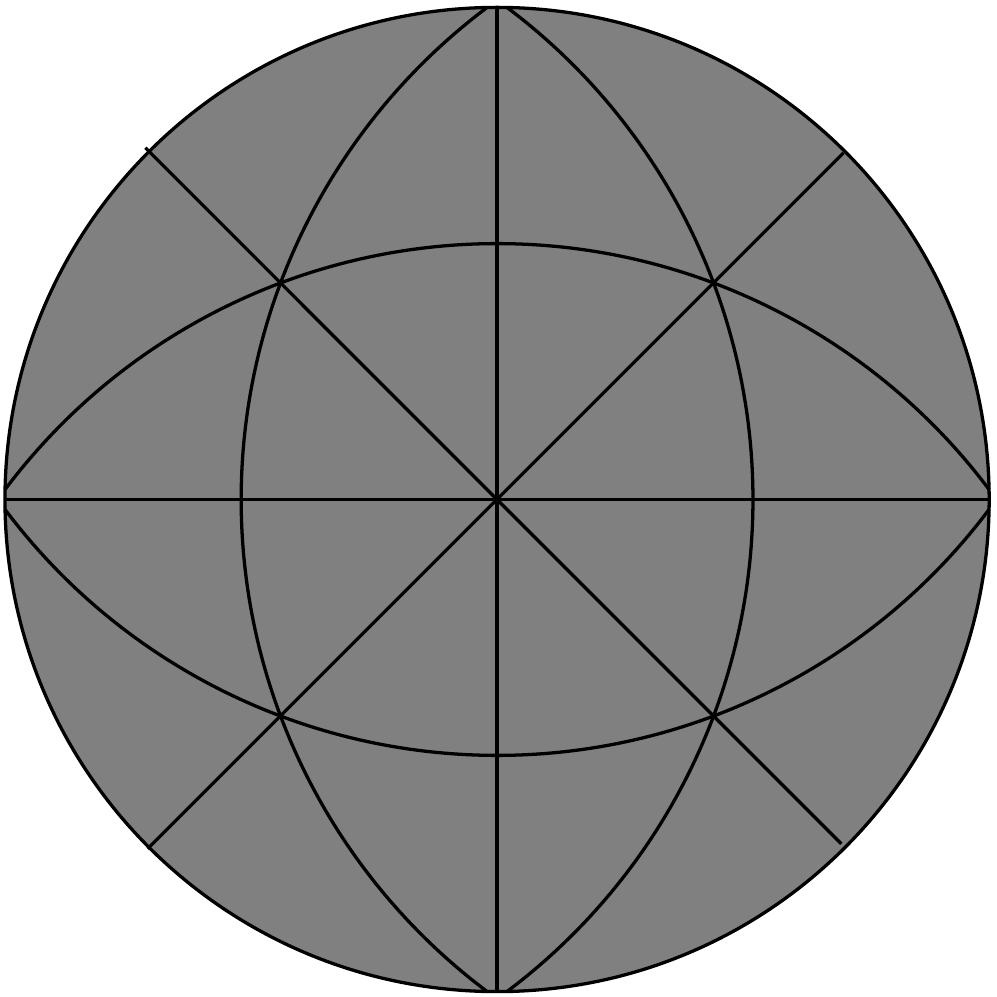} & \includegraphics[height=5cm]{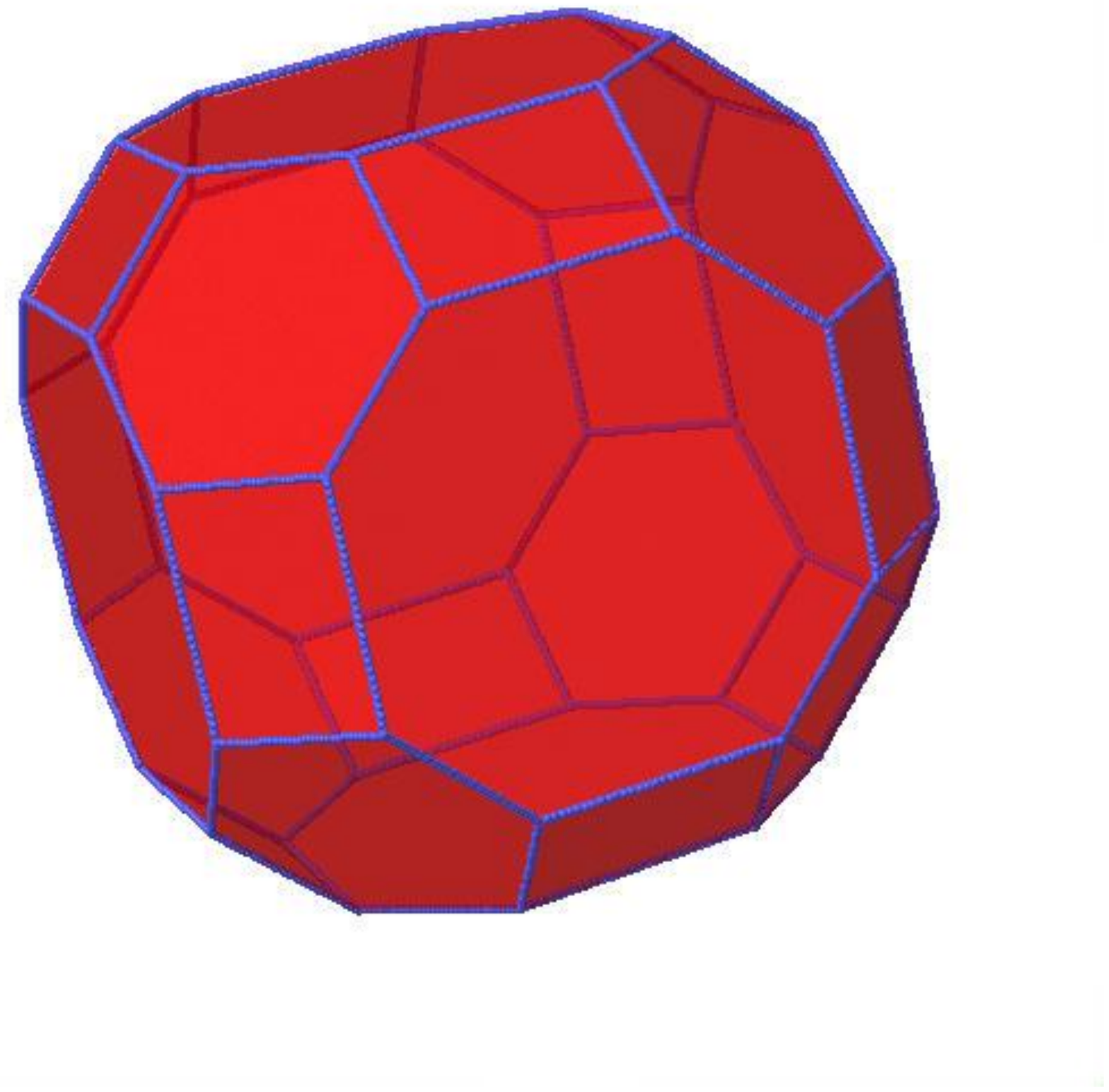}
\end{tabular}
\caption{Left - $B_3$ Coxeter Complex, Right - $B_3$ Permutahedron}
\label{permuteahedron}
\end{center}
\end{figure}

We show a much deeper correspondence between the Coxeter complex and the $Perm(W)$. Given $\mathscr{W}_{n,k}$, there is a subcomplex of 
$Perm(W)$ homotopy equivalent to the complement. This construction holds for any subspace arrangement embedded in the Coxeter arrangement.
 To prove this result, we use the following specialization of Proposition 3.1 in \cite{bjorner-ziegler}.

\begin{proposition}
\label{prop:dualcomplex}

Let $\Delta$ be a simplicial decomposition of the $k$-sphere, and let $\Delta_0$ be a subcomplex of $\Delta$. 
Let $P$ be the face poset of $\Delta$, and let $P_0$ be the lower order ideal generated by $\Delta_0$. 
Then $|\Delta| \backslash |\Delta_0|$ is homotopy equivalent to a regular CW complex $X$, 
and moreover, the face poset of $X$ is $(P\backslash P_0)^*$, where $*$ denotes taking the dual poset.

\end{proposition}
Thus we have the following corollary.

\begin{corollary}
 \label{cor:complex}
There is a subcomplex, $Perm_k(W)$ of $Perm(W)$ such that $\mathcal{M}(\mathscr{W}_{n,k}) \cong Perm_k(W)$. Moreover, the faces of 
$Perm_k(W)$ correspond to cosets $wW_I$ such that for all $J \subset I$, $W_J$ is not $k$-parabolic.
\end{corollary}
\begin{proof}
 We see that $\mathcal{M}(\mathscr{W}_{n,k}) \cong |\Delta(W)| \setminus |\Delta_k(W)|$, which by the previous theorem is equivalent to some regular CW complex $Perm_k(W)$. The face poset of $Perm_k(W)$ follows from Proposition \ref{prop:dualcomplex} and the description of the face poset of $\Delta_k(W)$. Since regular CW complexes are determined by their face poset, we note that face poset of $Perm_k(W)$ corresponds to a subcomplex of $Perm(W)$, and hence $Perm_k(W)$ may be viewed as a subcomplex of $Perm(W)$.
\end{proof}

Note that in a previous paper \cite{full-version}, we proved this result only for the $2$-skeleton of $Perm_k(W)$.

\begin{remark}
 When $k=3$, $\mathscr{F}(Perm_k(W))$ consists only of cosets $wW_I$, where the reflections in $I$ commute. It is not hard to see that the corresponding face of $Perm_3(W)$ is an $|I|$-cube. Thus, $Perm_3(W)$ is a cubical complex, that is, a polyhedral complex whose faces are all cubes. We need this fact for Section \ref{sec:optimality}.
\end{remark}

Naturally, we would like a set of representatives for our cosets.
First, we recall some combinatorics of Coxeter groups, as this will give us nice choices for representatives. 
Given an element $w \in W$, 
let $\ell(w)$ denote the minimum number of simple reflections $s_1, \ldots, s_k$ 
such that $w = s_1 \cdots s_k$. We refer to $\ell(w)$ as the \emph{length} of $w$. 
Given a simple reflection $s$, we call $s$ a \emph{descent} of $w$ if $\ell(ws) < \ell(w)$. 
The \emph{right weak order} on $W$ is defined as follows: given two elements $u, v \in W$, we say that $u \leq v$ if there exists $w \in W$, $v = uw$, and $\ell(v) = \ell(u) + \ell(w)$. 

\begin{figure}[htbp]
\begin{center}
 \includegraphics[height=7cm]{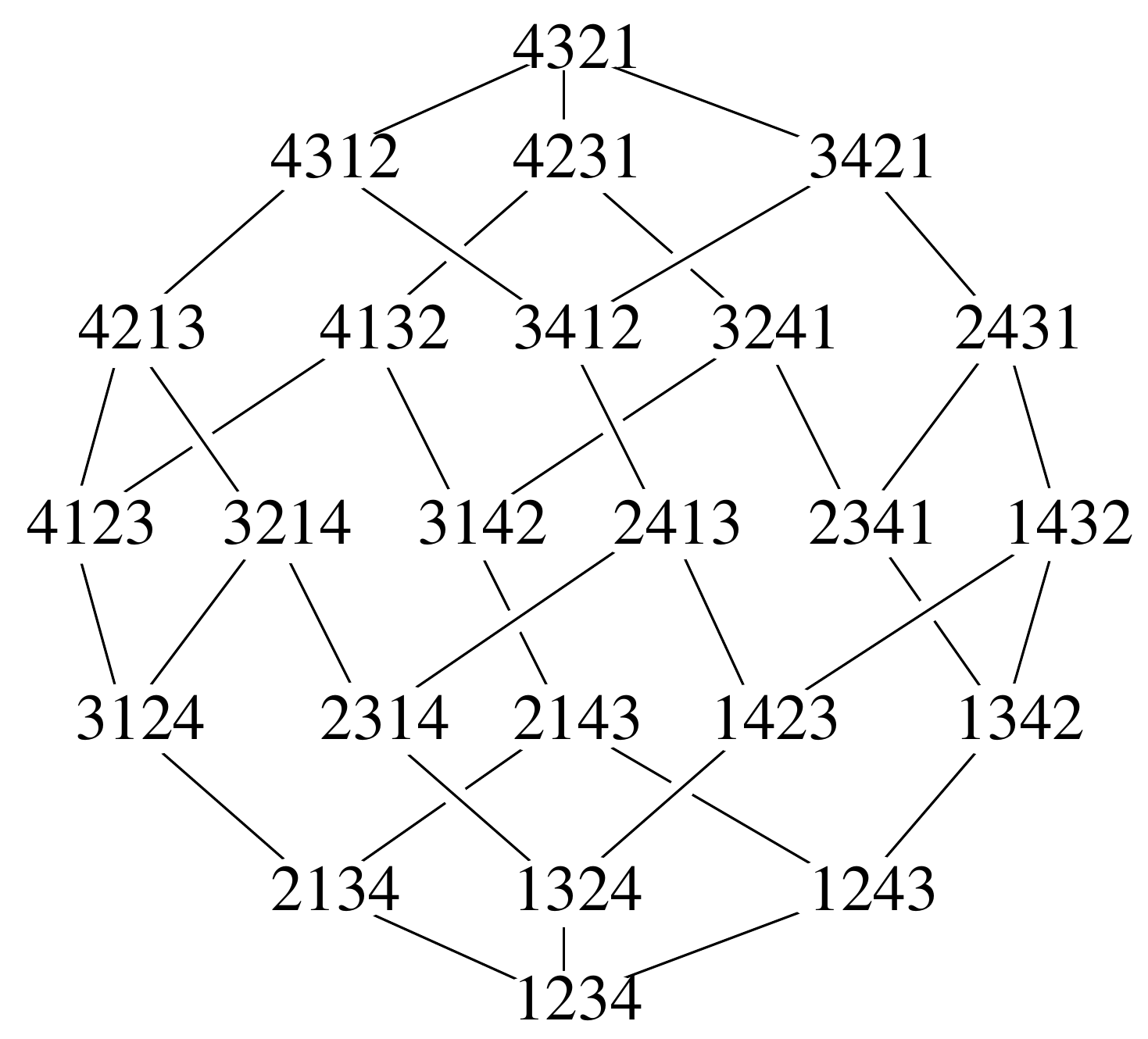}
\caption{$S_4$ under right weak order}
\label{S4rightweak}
\end{center}
\end{figure}
Figure \ref{S4rightweak} is from Aguiar and Sottile \cite{aguiar-sottile}.

When $W$ is finite, the right weak order has a maximum element, denoted $w_{0}$.
We now extend the definition of descent to finite standard parabolic cosets. Given $I \subset S, w \in W$, there is a unique element $w' \in wW_I$ of minimal length. We call this the coset representative of minimal 
length. Let $W^I$ denote the set of coset representatives of minimal length for $W_I$. Then the following results may be found in Humphreys' book \cite{coxbook}
\begin{theorem}[Proposition 1.10c in \cite{coxbook}]

 \begin{enumerate}
  \item $W^I = \{w \in W: \ell(ws) > \ell(w), \mbox{for all } s \in I \}$
 \item For all $w \in W$, there exists unique $u \in W_I, v \in W^I$ such that $w = uv$, and $\ell(w) = \ell(u) + \ell(v)$.
 \end{enumerate}

\end{theorem}

If a coset $wW_I$ is finite, then it also has a coset representative of maximal length, given by multiplying the minimal length representative on the right by the maximum element of $W_I$. 
We say that an element $s \in S \setminus I$ is a \emph{descent} for a coset $wW_I$ if and only if it is a \emph{descent} for the maximal length representative of $wW_I$. 
In this paper, we use coset representatives of maximal length. Also, the right weak order is used when applying the Cluster Lemma and techniques from discrete Morse theory.

\section{Discrete Morse Theory and Shellability}
\label{sec:discmorse}

Here we review the terminology used with discrete Morse theory, as well as state the major theorems we use. 
Throughout, let $P, Q$ be finite posets. There are several wonderful introductions to discrete Morse theory: we highly recommend 
the book by Jonsson \cite{jonsson}, which has several examples of the application of discrete Morse theory. Our terminology comes from \emph{Combinatorial Algebraic Topology} by Kozlov \cite{kozlov}.
 However, the results of this section are due to Forman \cite{formancell}, who used different (but equivalent) terminology.
For the reader familiar with discrete Morse theory, we note that we actually need the full power of the fundamental theorem of discrete Morse theory. That is, our complexes are not simplicial, so we need to the regular CW complex version of discrete Morse theory. Moreover, we have to compute the boundary operator of the resulting Morse complex when $k=3$. Finally, our Morse matchings will be constructed out of matchings arising from shellable simplicial complexes.
We begin by presenting the definition of an acyclic matching that appears in the literature.

\begin{definition}[\cite{kozlov}, Definition 11.1]
Let $P$ be a poset. 

 A matching in $P$ is a partial matching in the underlying graph of the Hasse diagram of $P$, i.e., it is a subset $M \subseteq P \times P$ such that 
\begin{itemize}
 \item $(a,b)\in M$ implies $b \succ a$ ($b$ covers $a$);
\item each $a \in P$ belongs to at most one element in $M$. 
\end{itemize}

When $(a,b) \in M$ we  write $a =d(b)$ and $b=u(a)$. 
A partial matching on $P$ is called acyclic if 

there does not exist a cycle \[b_1 \succ d(b_1) \prec b_2 \succ d(b_2) \prec \cdots \prec b_n \succ d(b_n) \prec b_1 \] with $n>2$ and all $b_i \in P$ being distinct.

\end{definition}

\begin{theorem}[\cite{kozlov}, Theorem 11.13]
\label{thm:fundthm}
Let $\Delta$ be a polyhedral complex, and let $M$ be an acyclic matching on $\mathcal{F}(\Delta)\setminus \{\hat{0}\}$. Let $c_i$ denote the number of critical $i$-dimensional cells of $\Delta$. 
\begin{itemize}
 \item[(a)] If the critical cells form a subcomplex $\Delta_c$ of $\Delta$, then there exists a sequence of cellular collapses leading from $\Delta$ to $\Delta_c$. 
\item[(b)] In general, the space $\Delta$ is homotopy equivalent to $\Delta_c$, where $\Delta_c$ is a CW complex with a bijection between the set of $i$-cells of $\Delta_c$ and $C_i$.
\item[(c)] Moreover, under this bijection $f$, for any two cells $\sigma$ and $\tau$ of $\Delta_c$ satisfying dim $\sigma =$ dim $\tau + 1$, the incidence number 
$[\tau : \sigma]$ is given by \[ [\tau : \sigma] = \sum_c \omega(c). \]
Here the sum is taken over all alternating paths $c$ connecting $f(\sigma)$ with $f(\tau)$, i.e., over all sequences $c = (f(\sigma), a_1 , u(a_1 ), \ldots , a_t , u(a_t), f(\tau) )$ such that 
$f(\sigma) \succ a_1$, $u(a_t ) \succ f(\tau)$, and $u(a_i )\succ  a_{i+1}$ , for $i = 1, . . . , a_{t-1}$. For such an alternating path, the quantity $\omega(c)$ is defined by
    \[ \omega(c) := (-1)^t [a_1 : f(\sigma)][f(\tau) : u(a_t )]\prod_{i=1}^{t}[a_i : u(a_i )] \prod_{i=1}^{t-1}[a_{i+1} : u(a_i)] \] 
where the incidence numbers in the right-hand side are taken in the complex $\Delta$.
\end{itemize}
\end{theorem}

Given an acyclic matching $M$, we say that a matching is \emph{optimal} if $\Delta_c$ is a minimal cell complex. 
Constructing an acyclic matching is often a very challenging problem, so we need to use the following result, 
known as the Cluster Lemma or Patchwork Theorem, which 
allows us to create an acyclic matching on a poset $P$ by piecing together acyclic matchings on the fibers of a poset map from $P$ to another poset $Q$. 

\begin{lemma}[\cite{kozlov}, Theorem 11.10]
\label{lem:clusterlem}
Assume that $\varphi:P \rightarrow Q$ is an order-preserving map, and assume that we have acyclic matchings on subposets $\varphi^{-1}(q)$ for all $q \in Q$. Then the union of these matchings is itself an acyclic matching on $P$.  
\end{lemma}

Using the Patchwork Theorem, we show how to associate an optimal matching to a shellable simplicial complex $\Delta$. This is 
already a known result, and is mentioned in Kozlov \cite{kozlov}. However, we make this result explicit, as our optimal matching on $Perm_k(W)$ is constructed by using shellability and the Patchwork Theorem.

Let $\Delta$ be an abstract simplicial complex. Given a face $\sigma$, let $\bar{\sigma} = \{ \tau: \tau \subseteq \sigma \}$.
 Recall that $\Delta$ is \emph{shellable} 
if its maximal simplices can be arranged in a linear order $F_1, \ldots, F_r$ so that, for all $1\leq i \leq r$, $\Delta_i \cap \bar{F}_i$ is pure and has dimension $\dim(F_i) - 1$, where $\Delta_i = \cup_{j < i} \bar{F}_j$. Such an order is called a \emph{shelling order}.
The definition of shellability for pure simplicial complexes is due to Bruggesser and Mani \cite{bruggesser-mani}, and was extended to 
nonpure simplicial complexes by Bj\"orner and Wachs \cite{bjorner-wachs}. 
 
An equivalent definition is the following: For every $1 \leq i < j \leq r$, there exists $1 \leq k < j$ such that $F_i \cap F_j \subseteq F_k \cap F_j$, and $|F_k \cap F_j| = |F_j| - 1$.
Given a maximal simplex $F_i$, we say it is spanning if $\Delta_i \cap \bar{F}_i = \bar{F}_i \setminus \{F_i\}$, that is, if $F_i$ is being attached by its entire boundary. 
One of the nice results regarding shellable complexes is that their homology groups, and homotopy type are both easy to describe. 
However, for this paper, we only use the fact that shelling orders give rise to optimal matchings.

To define an example of such a matching, we need to recall the definition of the restriction map.
Given a facet $F_i$, let $\mathscr{R}(F_i) = \{x \in F_i: F_i \setminus \{x\} \in \Delta_i \}$. 
We call $\mathscr{R}$ the \emph{restriction map}. The next lemma is essentially due to Bj\"orner and Wachs \cite{bjorner-wachs}: our 
novelty is using the terminology of order-preserving maps to state their result.

\begin{lemma}
\label{lem:patchforgraphs}
 Let $f: \mathscr{F}(\Delta) \to [r]$ be given by $f(\sigma) = \min \{i: \sigma \subseteq F_i \}$. Then $f$ is an order preserving map. Moreover, given $i \in [r]$, $f^{-1}(i) = [\mathscr{R}(F_i), F_i]$, 
where $\mathscr{R}$ is the restriction map.
\end{lemma}

\begin{theorem}
\label{thm:matchfromshell}
 Let $\Delta$ be a shellable complex with shelling order $F_1, \ldots, F_r$, and restriction map $\mathscr{R}$. Then
\begin{enumerate}
 \item There are optimal acyclic matchings on $\mathscr{F}(\Delta)$.
\item In such a matching $M$, there is one critical 0-cell.
\item In such a matching $M$, $C_k$, the set of critical $k$-cells $(k > 0)$ correspond to the set of facets $F_i$ such that $\mathscr{R}(F_i) = F_i$, and $\dim F_i = k$.
\item Given such a matching $M$, $\Delta \cong \bigvee_{\sigma \in C_i} \mathbb{S}^{\dim \sigma}$

\end{enumerate}

\end{theorem}

\begin{proof}
 By the Patchwork Theorem, we know we need to find an acyclic matching on the fibers of $f$, the map defined in Lemma \ref{lem:patchforgraphs}. 
However, the fibers are Boolean intervals. Let $i \in [r]$ such that $F_i$ is not spanning, and fix $x \in F_i \setminus \mathscr{R}(F_i)$. 
Then consider the map $g_{x,i}: [\mathscr{R}(F_i), F_i] \to [\mathscr{R}(F_i), F_i]$ given by 
\begin{displaymath}
 g_{x,i}(\sigma) = \left\{
\begin{array}{cc}
 \sigma \setminus \{x\} & \mbox{ if } x \in \sigma \\ \sigma \cup \{x\} & \mbox{ else }
\end{array}
\right.
\end{displaymath}

The map $g_{x,i}$ is clearly an involution, and gives an acyclic matching 
on $[\mathscr{R}(F_i), F_i]$. The union of these matchings is acyclic, and clearly has the properties stated in the theorem.

Note that the critical cells all correspond to facets. Thus, the resulting Morse complex is actually a wedge of spheres. \end{proof}

In our case, we always have a linear order on the vertex set $V(\Delta)$, so we can specify the map $g_{x,i}$ in the proof of this theorem by $x = \min F_i \setminus \mathscr{R}(F_i)$. 

\section{Generalization of Independence Complex of a Graph}
\label{sec:indcom}
In this section, we define a simplicial complex which generalizes the Independence complex of a graph $G$. 
We show that this complex, $Ind_k(G)$, is shellable when $G$ is a forest. This shelling order is used to construct an optimal matching for $Perm_k(W)$ in the next section. 
The complex $Ind_k(G)$ has vertex set $V(G)$, and simplices $\sigma$ correspond to vertex sets such that every component of $G[\sigma]$ has size at most $k$. Recall that $G[\sigma]$ is the \emph{induced} subgraph. That is, $V(G[\sigma]) = \sigma$, and $ij \in E(G[\sigma])$ if and only if $i, j \in \sigma$ and $ij \in E(G)$.
The case $k=1$ is the usual Independence complex studied in the literature. For more about the topology of $Ind_1(G)$, we invite the reader to consult Engstr\"om's paper \cite{engstrom}.

Given a tree $T$ with a root $r$, a \emph{tree-compatible} ordering is a linear order on $V(T)$ such that, given two vertices $u$ and $v$, 
if $v$ is contained on the unique path from $u$ to $r$, then $u \leq v$. Equivalently, a tree-compatible ordering is a linear extension of the partial order that is dual to what is known as the Tree order. Given a forest $F$, with a set of vertices $r_1, \ldots, r_k$, we 
define a tree-compatible ordering to be a linear order that is a tree-compatible order when restricted to each component. Given a tree-compatible 
ordering on $F$, it turns out that the lexicographic ordering of facets of $Ind_k(F)$ is a shelling order.

We claim the following:
\begin{theorem}
 \label{thm:shellability}
Let $F$ be a forest on $n$ vertices, and let $1 \leq k \leq n$. Then $Ind_k(F)$ is shellable. Consider a set of roots for $F$, and a tree-compatible ordering on $F$. 
Then a shelling order is given by lexicographic ordering on facets: $\sigma < \tau$ if $\min (\sigma \setminus \tau) \cup (\tau \setminus \sigma) \in \sigma$.
\end{theorem}

\begin{proof}

Let $v_1, \ldots, v_n$ be a tree-compatible order on $F$. Order the maximal simplices of $Ind_k(F)$ lexicographically. We claim that this is a shelling order. Let 
$F_1, \ldots, F_r$ denote the maximal simplices in this order. We use the phrase `large component' to mean a component with more than $k$ vertices.

Let $i, j$ be such that $F_i < F_j$. Let $x = \min F_i \setminus F_j$, and let $C$ be the component of the subgraph of $F$ induced by $F_j + x$ which contains the vertex $x$. Since $F_j$ is a facet, 
$|C| > k$. Since $F_i$ is also a facet, $C \setminus F_i \neq \emptyset$. Let $y = \min C \setminus F_i$. We show that $F_j + x - y$ does not contain a large component. 

Let $C_{>y}$ denote vertices of $C$ that are greater than $y$, and let $C_{< y }$ denote vertices of $C$ that are less than $y$. It suffices to show that the only vertex $v > \max C_{<y}$ that is adjacent to some vertex in $C_{<y}$ is the vertex $y$ itself. Then $F_j+x-y$ cannot have a large component, as such a large component would have to be in $F_j$ or $F_i$. 

There is some component $T$ of $F$ containing $C_{<y}$. Moreover, this tree $T$ has a root vertex $r$. Suppose there are vertices $u, v$ such that $u \in C_{<y}$, $v > y$, and $uv$ is an edge. Since $v > u$ and $uv$ is an edge, $v$ is on the unique path from $r$ to $u$. Since $C$ is connected, and $y > u$, $y$ also lies on the unique path from $r$ to $u$.  However, we see that $y$ lies on the unique path from $r$ to $v$, and hence $y > v$, a contradiction. Therefore, $C_{< y}$ has no edge to any vertex of $V_{> y}$.

Let $F_k$ be any facet containing $F_j + x - y$. Then clearly $F_k < F_j$, and $F_i \cap F_k \subseteq F_k \cap F_j = F_j - y$. Therefore, we have a shelling order.\end{proof}

Naturally, given the fact that we have a shelling order on $Ind_k(F)$, it would be nice to classify the homotopy type. 
Also, since we use these shelling orders to give matchings in the next section, classifying the critical cells is necessary.
Given a graph $G$ with a linearly ordered vertex set, and a subgraph $H$, let $\mathscr{C}$ be the components of $H$. Let $C$ be a component of $H$. 
Recall from the introduction that $N_{\mathscr{C}}^{<}(C)$ is the set of vertices $v$ in $G - H$ that are adjacent to some vertex in $C$, and such that $v < \min C$ in the linear ordering.
\begin{theorem}
\label{thm:spanningcells}
Given a tree-compatible order on a forest $F$, spanning simplices of $Ind_k(F)$ are simplicies $\sigma$ such that:
\begin{enumerate}
 \item $F[\sigma]$ consists of $t$ components, $\mathscr{C} = \{C_1, \ldots, C_t\}$, each of size $k$, where $0 \leq t \leq \frac{n}{k}$.
\item For every component $C$ of $F[\sigma]$, we have $N_{\mathscr{C}}^{<}(C) \neq \emptyset$
\item Every vertex in $V(F) \setminus \sigma$ is adjacent to some component $C \in \mathscr{C}$.
\end{enumerate}

\end{theorem}

\begin{proof}
 Clearly, if a subset $I$ has all the stated properties, then it is a facet, and $\mathscr{R}(I) = I$. So suppose we have a facet $I$ 
such that $\mathscr{R}(I) = I$. Let $x = \min V(F)$. If $x \in I$, then $x \not\in \mathscr{R}(I)$ by definition of the restriction map. Thus $x \not\in I$. However, since $I$ is a facet, $I + x$ must contain a large component. Let $C$ be the lexicographically least large component (of size $k+1$). Then $x \in C$. Consider removing $N[C]$ from $V(F)$, obtaining a new set $V'$. Note that $x \in N^<_{\mathscr{C}}(C)$.

Now let $x = \min V'$. Again, we see that $x \not\in I$. Moreover, since $x \not\in N[C]$, $I+x$ contains a large component $C'$ of size $k+1$ that is disjoint from $C$. 
Given our choice of $x$, every element of $C'$ is greater than $x$. Choose $C'$ to be lexicographically least, and remove $N[C']$ from $V'$. Continuing in this manner, we see that $I$ is a disjoint
 union of components of size $k$, and for each component $C$ we have $N^<_{\mathscr{C}}(C) \neq \emptyset$. Finally every remaining vertex is adjacent to some component.
\end{proof}

\begin{example}
 Let $G$ be a graph with vertex set $\{1, \ldots, 9 \}$, and edge set $\{ 14, 24, 34, \\ 45, 58, 68, 78, 89\}$. Then the natural ordering $1 < 2 < \ldots < 9$ is a tree-compatible order, 
where $9$ is the root vertex. $Ind_3(G)$ has many facets for this complex. However, one can check that the spanning simplices correspond to vertex sets $234789$ and $458$. So $Ind_3(G)$ is homotopy equivalent to a 
wedge of two spheres, one of dimension $2$ and one of dimension $5$. In particular, $Ind_k(G)$ does not always correspond to a wedge of equidimensional spheres.
\end{example}

\section{Matching Algorithm and Main Results}
\label{sec:matchalg}
In this section, we define an optimal matching on $Perm_k(W)$, and prove the main theorems from the introduction. Given $W$ with simple reflections $S$, order $S$ so that we have a tree-compatible order on the Dynkin diagram $D(W)$ (where tree-compatible order is defined in the previous section).

Recall that elements of $Perm_k(W)$ correspond to parabolic cosets $wW_I$ that do not contain a coset $w'W_J$ where $W_J$ is $k$-parabolic. Given such a coset $wW_I$, suppose $w$ is of maximum length in $wW_I$. Finally, let $\mathscr{P}(wW_I, s) = \{J \subset I: W_{J+s} \mbox{is $k$-parabolic} \}$, and let $Des(w)$ be the descent set of $w$. Then we match $wW_I$ based on the following algorithm:

\textbf{Let} $L = S$.

 \textbf{While} $L \neq \emptyset$

\hspace{1cm} \textbf{Let} $s = \min L$

\hspace{1cm} \textbf{If} $s \not\in Des(w)$

\hspace{2cm}    \textbf{Set} $L = L - s$

\hspace{1cm} \textbf{Else If} $s \in I$

\hspace{2cm}     \textbf{Return} $wW_{I\setminus \{s \}}$

\hspace{1cm} \textbf{Else If} $wW_{I + s} \not\in Perm_k(W)$

\hspace{2cm} \textbf{Let } $J = \min \mathscr{P}(wW_I, s)$

\hspace{2cm}        \textbf{Set} $L = L - J - s$

\hspace{1cm}    \textbf{Else}

\hspace{2cm}        \textbf{Return} $wW_{I \cup \{s\}}$

\textbf{End While}

\textbf{Return} $wW_I$

Given a coset $wW_I$, we refer to the coset the algorithm outputs as $M(wW_I)$. We match $wW_I$ with $M(wW_I)$ if $M(wW_I) \neq wW_I$. 
Otherwise, $wW_I$ is \emph{critical}. Note that it is not entirely obvious that this is a matching. However, we show that the matching given by this algorithm is one arising from shellability of generalized independence complexes from the last section.

First, note that there is a natural order-preserving map $\varphi: \mathscr{F}(Perm_k(W)) \to W$, where $W$ is given right weak order. 
The map is given by sending a coset $wW_I$ to its maximal length representative $w$. Moreover, given $w \in W$, it is not hard to see that $\varphi^{-1}(w)$ is isomorphic to the face poset of $Ind_{k-2}(D_w)$, where $D_w = D[Des(w)]$. That is, given a coset $wW_I \in Perm_k(W)$ with maximal length element $w$, we have $I \subset Des(w)$, and the subgraph of $D$ induced by $I$ has no component of size $k-1$. Thus, we can conclude, that given a tree-compatible ordering on $S$, the techniques of the last section give us a collection of acyclic matchings $M_w$, one for each $w \in W$. Then the Patchwork Theorem gives us an acyclic matching $M'$ on $Perm_k(W)$.

Let $wW_I \in Perm_k(W)$ with maximum length element $w$. Then $I \in Ind_{k-2}(D_w)$. Let $F_j$ be the lexicographically first facet containing $I$. 
Then $M'(wW_I) = wW_{J}$ where $J$ is obtained from $I$ by adding or removing $x = \min F_j \setminus \mathscr{R}(F_j)$, depending on whether or not $x \in I$. If $F_j = \mathscr{R}(F_j)$, then we define $M'(wW_I) = wW_I$. We have thus constructed an involution coming from our matching.

\begin{theorem}
Let $M'$ be obtained as described in the above paragraph. Then $M = M'$.

\end{theorem}
\begin{proof}
 Let $wW_I$ be a coset with maximal length representative $w$, and suppose the while loop for the matching algorithm for $M$ runs $m$ times, 
and let $L_i$ be the list $L$ after running the while loop of the matching algorithm $i$ times.
Also, let $F_j$ be the lexicographically least facet of $Ind_k(D_w)$ containing $I$.

We claim that for each $i$, $1 \leq i \leq m$,  $F_j \subseteq (I \cup L_i)$, and $I \setminus L_i \subseteq \mathscr{R}(F_j)$. Let $s_i = \min L_i$.
Suppose at step $i$, $s_i \not\in I$, and the algorithm chooses not to add it to $I$. Then either $s_i \not\in Des(w)$, or $W_{I + s_i}$  
contains a $k$-parabolic subgroup.
Thus $s_i \not\in F_j$, so we have $F_j \subseteq (I \cup L_{i+1})$. Also, if $s_i \not\in Des(w)$, 
then $I \setminus L_{i+1} \subseteq \mathscr{R}(F_j)$.
Suppose $W_{I + s_i }$ contains a $k$-parabolic subgroup. Then so does $W_{F_j + s_i}$. Let $J$ be minimal such that $W_{J}$ is a $k$-parabolic subgroup contained in $W_{F_j + s_i}$. Note that $s_i \in J$.
 We claim that $J_{> s_i} \subseteq \mathscr{R}(F_j)$. This is clear if $J_{> s_i} = \emptyset$. Otherwise, fix $y \in C_{> s_i}$. If $K = F_J + s_i - y \in Ind_{k-2}(D_w)$, then $y \in \mathscr{R}(F_j)$.
If not, then $uW_K$ contains a $k$-parabolic subgroup $W_{J'}$,
which can be chosen to be minimal with respect to the condition $s_i \in J'$. We claim that $J' < J$. Since we have a tree-compatible ordering, there are no edges between $S_{> y}$ and $S_{<y}$. Since $J'$ induces a connected subgraph of $D$, not involving $y$, it follows that $J' \subset S_{< y}$. Since $y \in J$, and $|J'| = k-1$, we must have $\min J' < \min J$. However, this is a contradiction, since $s_i \in J'$, and $J$ was chosen to be lexicographically minimal. Therefore $K \in Ind_k(D_w)$, and hence $y \in \mathscr{R}(F_j)$ for all $y \in C_{> s_i}$. 
Thus $I \setminus L_{i+1} \subseteq \mathscr{R}(F_j)$.
Therefore we have proven the claim for each $i$ by induction.

Now we would like to show that when the algorithm terminates, $M(wW_I) = M'(wW_I)$.
Clearly, if $M(wW_I) = wW_I$, we see from our claims that $\mathscr{R}(F_j) = I = F_j$, whence $wW_I$ is left unmatched in the union of acyclic matchings, so $M'(wW_I) = wW_I = M(wW_I)$.
Otherwise, suppose the matching algorithm terminated after $m$ steps, and matched $wW_I$ to another coset, by either removing or adding a reflection $s$ to $I$. 
If the algorithm terminates by adding a reflection $s$ to $I$, we see that $s \in F_j$. 
By our properties, 
$s \in F_j \setminus \mathscr{R}(F_j)$. Since this is the first reflection we could possibly remove, we see that $s$ is the minimum of $F_j \setminus \mathscr{R}(F_j)$, and thus $M(wW_I) = M'(wW_I)$.

Suppose instead that the algorithm removes some $s \in I$. 
Suppose $s \in \mathscr{R}(F_j)$. Then there exists $t$ such that $F_j + t - s$ is contained in a lexicographically smaller facet. 
Moreover, $D_w[F_j + t]$ has a component with at least $k-1$ vertices.
Let $J$ be minimum such that $J \subset F_j + t$ and $W_J$ is $k$-parabolic. If $\{s, t\}$ is not a subset of $J$, then we obtain a contradiction to the fact that 
$uW_{F_j}, uW_{F_j + t - s} \in Perm_k(W)$. Thus, $\{s, t\} \subseteq J$, and must be in the same component of $J$. However, since $J$ was chosen to be minimal, 
when the algorithm studied $t$, it would have removed $J$, and thus $s$, from $L$. Therefore we have a contradiction, and $s \not\in \mathscr{R}(F_j)$. Again,
by our previous claims, we can conclude that 
$s \in F_j \setminus \mathscr{R}(F_j)$. It is also the first reflection encountered that could be removed, and so $s$ is the minimum of $F_j \setminus \mathscr{R}(F_j)$, and thus $M(wW_I) = M'(wW_I)$. \end{proof}

\begin{proposition}
\label{prop:critcells}
Let $wW_I \in Perm_k(W)$ with maximum length element $w$. If $wW_I$ is critical, then the following hold:
\begin{enumerate}
\item $I \subset Des(w)$.
 \item $D[Des(w)]$ consists of $t$ components, $\mathscr{C} = \{C_1, \ldots, C_t \}$, each of size $k-2$
\item For every component $C$ of $D[I]$, we have $N^<_{\mathscr{C}}(C) \cap Des(w) \neq \emptyset$
\item For every $v \in Des(w) \setminus I$, there exists a component $C$ such that $v \in N(C)$.
\end{enumerate}

\end{proposition}
\begin{proof}
 The result follows from the poset map $\varphi$, and the description of critical cells given in the last section.
\end{proof}

Note then that $c_i$, the number of critical cells of dimension $i$, is $0$ unless $i = t(k-2)$ for some $0 \leq t \leq \frac{n}{k}$. 
\begin{lemma}
 The matching $M$ is optimal.
\end{lemma}
\begin{proof}
 Suppose $k > 3$. Then the fact that there are only critical cells of dimension $t(k-2)$ for $0 \leq t \leq \frac{n}{k}$ implies that 
the boundary operator of $\Delta_c$, the complex given in Theorem \ref{thm:fundthm}, must be the $0$-map. Hence the cellular chain groups of $\Delta_c$ are isomorphic to the homology groups. The case $k=3$ is proven below in Section \ref{sec:optimality}, and is considerably more involved.
\end{proof}

\begin{proof}[Proof of Theorems \ref{thm:minimality}, \ref{thm:homology}, \ref{thm:bettis}]
 Immediate from the previous Lemma, Proposition \ref{prop:critcells}, Theorem \ref{thm:fundthm}, and the observation about the number of critical cells of a given dimension.
\end{proof}

\section{Betti Numbers for Irreducible Coxeter Groups}
\label{sec:bettis}

\subsection{The $k$-Equal Arrangement}

We can now describe the algorithm for the matching $M$ using terminology from set compositions.
Given an element $(B_1, B_2, \ldots, B_t)$ we consider pairs of adjacent blocks $B_i$ and $B_{i+1}$. We start with $i=1$. 
\begin{enumerate}
\item If $B_i$ is not a singleton, we match \[(B_1, \ldots, B_i, B_{i+1}, \ldots, B_t)\] with \[(B_1, \ldots, B_{i-1}, \{\max(B_i)\}, B_i \setminus \{\max(B_i)\}, \ldots, B_t).\] 
\item If there is a ascent from $B_i$ to $B_{i+1}$, we set $i = i+1$ and start over at step one.
\item If $|B_{i+1}|=k-1$, then we set $i=i+2$ and start over at step one. 
\item We match \[(B_1, \ldots, B_i, B_{i+1}, \ldots, B_t)\] with the element \[(B_1, \ldots, B_i\cup B_{i+1}, \ldots, B_t).\] 

\end{enumerate}

An example of the matching along with some critical elements is shown in Figure \ref{fig:matching}. 
\begin{figure}[htbp]
  \begin{center}
    \includegraphics[width=0.75\textwidth]{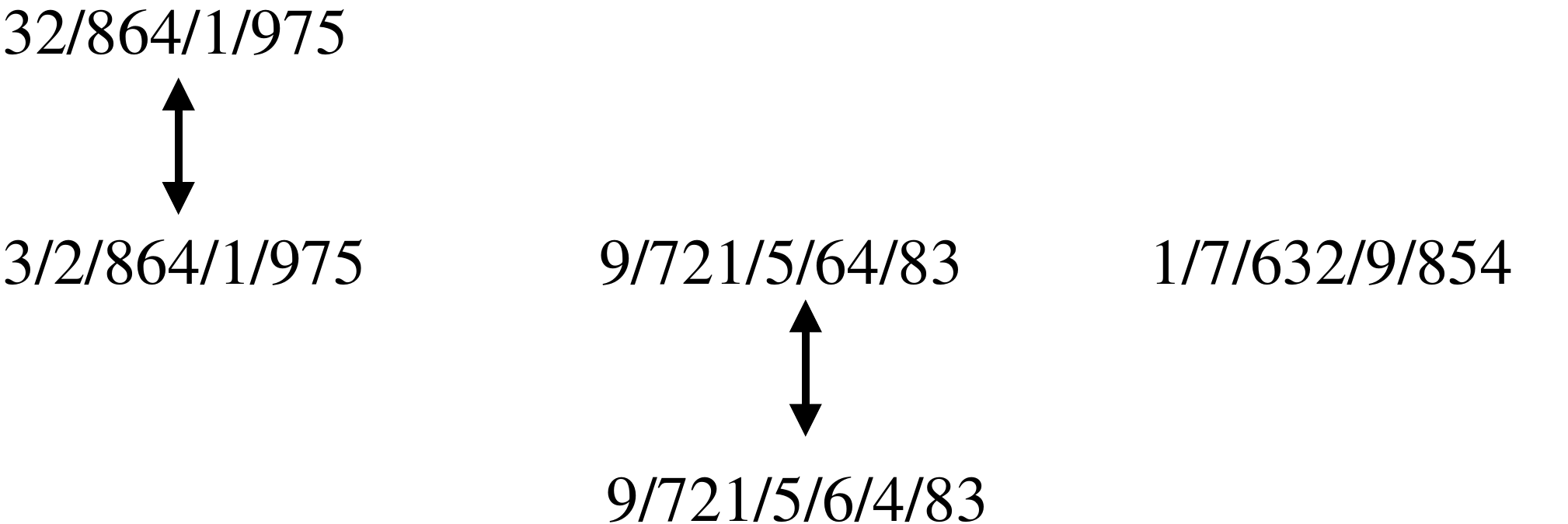}
    \caption{A matching between elements in $Perm_4(A_8)$}
    \label{fig:matching}
  \end{center}
\end{figure}

It remains to compute the number of critical elements. A weak integer composition of $n$ is a sequence of nonnegative integers $\mu = (\mu_1, \ldots, \mu_k)$ such that $\mu_1 + \mu_2 + \cdots + \mu_k = n$. 
We refer to $\ell(\mu) = k$ as the length of $\mu$, and $|\mu| = \sum_{i=1}^{\ell(\mu)} \mu_i$. We use $\mu \models n$ to say that $\mu$ is a weak integer partition with $|\mu| = n$. Given a weak integer compositions $\mu$, let $\binom{n}{\mu} = \binom{n}{\mu_1, \ldots, \mu_k}$.
Let $t$ be an integer such that $1 \le t \le n/k$ . Then the number of unmatched cells in dimension $t(k-2)$ is given by:

\[\sum_{\substack{\mu \models n+1 \\ \ell(\mu) = t+1 \\ \mu_m \geq k, \forall 1 \leq m \leq t}} \binom{n+1}{\mu} \prod^{t}_{m=1} \binom{\mu_m-1}{k-1} \]

where the sum is over all integer compositions of $n+1$ into $t + 1$ parts, such that each part, with the exception of the last part, has size at least $k$. 
The formula comes from the following: consider a composition $A_1, \ldots, A_{t+1}$ of $[n+1]$ into $t + 1$ parts whose sizes are given by $\mu_1, \ldots, \mu_{t+1}$. 
For each part, besides the last one, take $k-1$ elements $x$ that are not the maximum of that part. Make this a block, and place all other elements of $A_i \setminus X$ 
as singletons in increasing order before $X$, to get a set composition $\mathscr{A}_i$ that consists of singletons, and ends with a block of size $k-1$.
 Finally, partition $A_{t+1}$ into singletons and place them in increasing order to obtain a set composition $\mathscr{A}_{t+1}$. Then let $C$ be given by 
starting with the blocks of $\mathscr{A}_1$ in order, followed by the blocks of $\mathscr{A}_2$ in order, and so on. This creates a critical set composition.
Clearly this gives all set compositions that meet our criteria for not being matched.
Thus we have successfully computed $\widetilde{\beta}_{t(k-2)}(\mathcal{M}(\mathscr{A}_{n,k}))$.
We note that this formula was also found by Peeva, Reiner and Welker \cite{peeva-reiner-welker}.

\subsection{The Signed $k$-equal Arrangement}
For type $B$, one can use type $B$ set compositions to understand our matching. A type $B$ set composition is a sequence $C = (B_0, \ldots, B_k)$ of disjoint subsets of $\{\bar{n}, \bar{n-1}, \ldots, \bar{1}, 1, 2, \ldots, n \}$, such that $0 \in B_0$, and for each $i$, either $i$ is in some block of $C$, or $\bar{i}$ is, but not both. Moreover, we require elements in $B_0$ to be unbarred.

\begin{figure}[htbp]

 \label{bmatching}
\scalebox{.6}{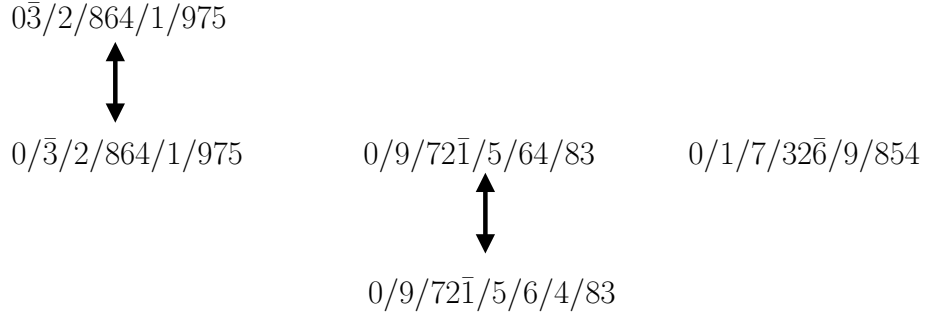}
\caption{An example of the matching for $B_9$ with $k=4$}

\end{figure}

Using our optimal matching, and discrete Morse theory, we obtain the following:
\begin{theorem}
 \label{thm:typeB}
$H_{t(k-2)}(\mathcal{M}(\mathscr{B}_{n,k}))$ is free abelian of rank 
\[\sum_{\substack{\mu \models n \\ \ell(\mu) = t+1 \\ \mu_m \geq k, \forall 1 \leq m \leq t}} \binom{n}{\mu} 2^{n-\mu_1+k-1} \prod_{m=1}^t \binom{\mu_m - 1}{k-1} + \sum_{\substack{\mu \models n \\ \ell(\mu) = t+1 \\ \mu_m \geq k, \forall 2 \leq m \leq t \\ \mu_1 = k-1}} \binom{n}{\mu} 2^{n-k+1} \prod_{m=2}^t \binom{\mu_m - 1}{k-1} \]
\end{theorem}
\begin{proof}
It suffices to count the number of cells of rank $t(k-2)$. In the first summation, we are summing over set compositions for which $B_1$ is a singleton. In these cases, $\mu_1$ corresponds to all the 
singletons leading up to the first $k-1$ block $B_i$, as well as the elements of $B_i$. We can place signs on every element except the singletons leading up to $B_i$. This explains the power of $2$ in the 
summation. The rest of the terms come from the same arguments as the type $A$ case.

The second summation is over critical set compositions for which $B_1$ is of size $k-1$. In this case, we know the elements of $B_1$ must be all negative, but we are still free to choose the signs of the remaining $n - k+1$ elements. 
This explains the power of $2$ in the second summation. The rest of the second summation again follows from counting arguments as in the type $A$ case.\end{proof}
We note that this result specializes to a formula given by Bj\"orner and Sagan \cite{subspacesBD}, when $t=1$.

\subsection{The Type $D$ $3$-Equal Arrangement}
Next we study the type $D$ $3$-equal arrangement. Note that since $\mathscr{D}_{n,k} = \mathscr{B}_{n,k}$ for $k > 3$, this is the only case left to study for classical reflection groups. Studying the Betti numbers again reduces to counting critical cosets. The proof of the following result is similar to type $A$ and $B$, so we omit the details.
\begin{theorem}
 \label{thm:typeD} 
$H_{t}(\mathcal{M}(\mathscr{D}_{n,3}))$ is free abelian of rank 
\[\sum_{\substack{\mu \models n \\ \ell(\mu) = t+1 \\ \mu_m \geq 3, \forall 2 \leq m \leq t \\ \mu_1 \geq 4}} \binom{n}{\mu} 2^{n-\mu_1+2} \prod_{m=1}^t \binom{\mu_m - 1}{2} + \sum_{\substack{\mu \models n \\ \ell(\mu) = t+1 \\ \mu_m \geq 3, \forall 2 \leq m \leq t \\ \mu_1 = 3}} \binom{n}{\mu} 7*2^{n-3} \prod_{m=2}^t \binom{\mu_m - 1}{2} \]
\end{theorem}

\subsection{The Finite Exceptional Coxeter Groups}
In this section, we detail methods for determining the Betti numbers for the remaining irreducible finite Coxeter groups. Let $W$ be an exceptional 
Coxeter group of rank $n$. One of our main results is that $Perm_k(W)$ has an optimal 
acyclic matching that only has unmatched elements in ranks that are multiples of $k-2$, and one unmatched element in rank $0$. When there is only one such rank 
satisfying these conditions, it is not hard to compute the Betti number. That is, one just computes the $f_i$ number of $i$-faces of $Perm_k(W)$, and then computes $\sum_{i=1}^n (-1)^i f_i$, the reduced Euler characteristic. The final result gives, up to sign, the rank of the only non-trivial homology group. 

We note that in some cases there are two ranks of nontrivial homology. By observation, this only occurs when $W$ is of type $E$. 
Clearly, after computing the Euler characteristic, determining one of the Betti numbers allows us to determine the other one. 
So of course, we have reduced our problem to the case of determining $\beta_{2(k-2)}$. In this case, we know that unmatched elements correspond to cosets $wW_I$ with maximum length 
representative $w$, where 
$I$ corresponds to a disjoint collection of components of size $k-2$ in the subgraph $D[Des(w)]$. Moreover, for each component $C$ there exists $x \in Des(w)$, $x < \min C$, such that 
$x \cup C$ is connected. Similarly, an unmatched element has some set of prescribed ascents as well. In other words, given sets $I$ and $J$ with $I \subseteq J \subseteq S$, 
we can determine if there exists an element $w$ such that $Des(w) = J$ and $wW_I$ is critical. 
Then it would remain to compute the number of elements $w$ with this given descent set $J$.

We determine necessary and sufficient conditions on $I$, $Des(w)$, and $S \setminus Des(w)$ such that 
a coset $wW_I$ with maximal length representative $w$ is unmatched. For a fixed $I$, let $D(I)$ and $A(I)$ be subsets of $S$ such that 
$wW_I$ is unmatched if and only if $D(I) \subset Des(w)$ and $A(I) \subset S \setminus Des(w)$ (given a particular choice of $I$, it is not too hard to determine $D(I)$ and $A(I)$). 
Let $\beta_I = |\{ w \in W: D(I) \subset Des(w) \mbox{ and } A(I) \subset S \setminus Des(w) \}|$. 
Finally, let $\mathscr{I}$ denote the set of all possible $I \subset S$ which are $2$ components, each of size $k-2$. Then $\beta_{2(k-2)} = \sum_{I \in \mathscr{I}} \beta_I$.
So it suffices to compute $\beta_I$. However, while counting the number of elements whose descent set contains a given set $J$ (by counting cosets of $W_J$), 
counting the number of elements with a set of prescribed descents and prescribed ascents involves using inclusion-exclusion, to restate the problem only in terms of enumerating elements with prescribed descents. 
Fix $I \in \mathscr{I}$, $T \subset A(I)$. Let $\mathscr{A}_{I,T} = \{wW_{T \cup D(I)}: w \in W \}$. As noted before, $\mathscr{A}_{I, T}$ corresponds to the number 
of elements $w \in W$ with $D(I) \cup T \subset Des(w)$. So, by inclusion-exclusion, $\sum_{T \subset A(I)} (-1)^{|T|} |\mathscr{A}_{I, T}|$ counts the number of elements $w \in W$ for which $D(I) \subset Des(w)$ and 
$A \subset S \setminus Des(w)$.

Thus we obtain $\beta_{2(k-1)} = \sum_{I \in \mathscr{I}} \sum_{T \subset A(I)} (-1)^{|T|} |\mathscr{A}_{I, T}|$. Naturally, this summation can be rather challenging to compute, although it is simple 
enough that it can be done by hand. For the remaining cases, this summation was used to determine $\beta_{2(k-2)}$. However, we omit the rather tedious computations. The resulting Betti numbers appear in the Appendix at the end of this paper.

\section{Proof of Optimality When $k=3$}
\label{sec:optimality}

 By Theorem \ref{thm:fundthm}, $Perm_3(W)$ is homotopy equivalent to some space 
$\Delta_c$ such that the $i$-cells in $\Delta_c$ are indexed by the unmatched $i$-cells of $Perm_3(W)$. We would like to show the boundary operator 
of $\Delta_c$ is the $0$-map, which would allow us to conclude that $\Delta_c$ is a minimal complex. 
Thus, we want to show the summation in Theorem \ref{thm:fundthm}, part c, 
is zero by constructing a sign-reversing involution on alternating directed paths between pairs of critical cells.
Most examples of discrete Morse theory in the literature have never had to use the boundary formula. 
Much 
like when $k > 3$, for most examples it is immediately clear that the boundary map is the $0$-map, because there are no cells in consecutive dimensions.

Given any coset $wW_I \in Perm_3(W)$, we can construct pairs of alternating directed paths $P$ and $Q$ with $\omega(P) + \omega(Q) = 0$, where both paths start at $wW_I$ and end with the same coset $vW_J$, and $w(P)$ is as defined in Theorem \ref{lem:clusterlem} part c.
In general, 
the ending coset $vW_J$ may not be critical. However, we show that given any alternating directed path between two critical cells, some subpath is identical to one generated by 
these algorithms. This fact is used to construct a sign-reversing involution on alternating directed paths between pairs of critical cells.

\begin{figure}[htbp]
 \begin{center}
  \includegraphics[height=6cm]{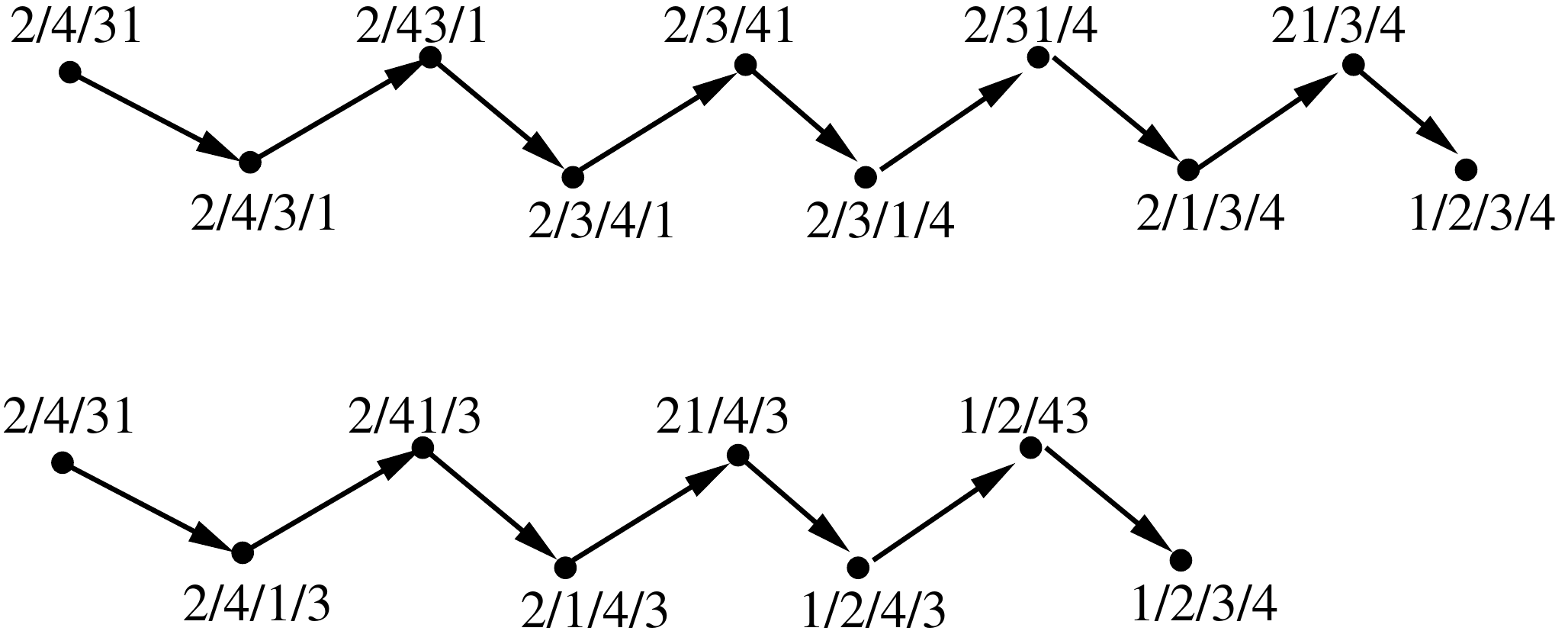}
 \end{center}
\caption{Examples of alternating directed paths in $A_3$}
\label{pathchaser}
\end{figure}

An example of paths coming from our construction is given in Figure \ref{pathchaser}. Recall that the faces of the $A_3$-permutahedron correspond to set compositions of $\{1, 2, 3, 4 \}$.
Figure \ref{pathchaser} has two alternating, directed paths that start and end with the same set compositions. We constructed these paths using an algorithm given below.

To make definitions easier, for this section we assume that when $W$ is irreducible, it is given one of the linear orders appearing in Figure \ref{fig:linearorder}. If $S$ decomposes as 
$S_1, \ldots, S_m$, where each of the $S_i$ are disjoint, and correspond to a connected component of $D(W)$, then we order the reflections so that every reflection of $S_i$ is less than every reflection of 
$S_j$, whenever $i < j$, and then order the reflections in each individual $S_i$ according to Figure \ref{fig:linearorder}. Again, the results actually hold for any tree-compatible order, however, the proofs and 
definitions are far more complicated. To keep the presentation simple, we shall only give the proofs for the cases where we have chosen the above linear orders. Finally, given a linear order on $S$, 
we extend it to a linear order on $S \cup \{\infty \}$ by making $\infty$ the unique largest element of the linear order.

Let $wW_I$ be a coset in $Perm_3(W)$ with maximal length representative $w$. In general, let $w_0(wW_I)$ denote the maximal length representative of a coset $wW_I$. 
If $M(wW_I) \neq wW_I$, we let 
$m(wW_I)$ denote the simple reflection that was added or removed when running the algorithm for $wW_I$. 
If $M(wW_I) = wW_I$, we let $m(wW_I) = \infty$. The elements of $I$ that are less than $m(wW_I)$ form a set $\{s_{i_1}, \ldots, s_{i_m} \}$ for some $m$. 
Moreover, these elements form an independent set of $D_w$, they each have a back neighbor in $D_w$, and all descents less than $m(wW_I)$ must be adjacent to some $s_{i_j} \in I$.
Given $s_i < s_j$, let $[s_i, s_j] = \{s_k: s_i \leq s_k \leq s_j \}$. Suppose $m > 1$. For $1 \leq j < m$, we let $A_{j} = [s_{i_j+2}, s_{i_{j+1}}]$. We define $A_0 = [s_1, \ldots, s_{i_1}]$, $A_m = [s_{i_m+2}, m(wW_I)]$. 
We call these \emph{ascending blocks}.

Given a coset $wW_I$ and a reflection $s \in I$, $s \leq m(wW_I)$, we construct two alternating directed paths that start with $wW_I$, and end at a coset 
$w'W_{I-s}$. The weights on these paths coming from Theorem \ref{thm:fundthm} part c cancel, and form the basis of our involution. In general, we show that any alternating directed path between two 
unmatched cosets must contain one of these constructed paths as a subpath. Then we define the involution by finding the first such subpath, and replacing it with its opposite construction.
We admit that this is a complicated involution.

Given a coset $wW_I$, let $s \in I$ such that $s \leq m(wW_I)$, and let $A$ be the ascending block containing $s$. The algorithm returns $\mathscr{P}$, a sequence of vertices of an alternating, directed path.

\textbf{Let} $J = I \setminus \{s \}$.

\textbf{Let} $u = w$.

\textbf{Let} $\mathscr{P} = (wW_I, uW_J)$.

 \textbf{While} $m(uW_J) \in A$

\hspace{1cm} \textbf{Let} $r = m(uW_J)$

\hspace{1cm} \textbf{Append} $M(uW_J)$ \textbf{to} $\mathscr{P}$

\hspace{1cm}    \textbf{Append} $urW_J$ \textbf{to} $\mathscr{P}$

\hspace{1cm} \textbf{Let} $u = w_0(urW_J)$.

\textbf{End While}

\textbf{Return} $\mathscr{P}$

The other algorithm only differs from the first algorithm by replacing the second line with \textbf{Let} $u = ws$.
Given a coset $wW_I$ and $s \in I$, $s \leq m(wW_I)$, let $p(wW_I, s)$ be the result of the first algorithm run with inputs $wW_I$ and $s$, and let $\hat{p}(wW_I,s)$ 
be the result of the second algorithm run with those inputs. Finally, let $\alpha(p(wW_I, s)) = \hat{p}(wW_I, s)$, and $\alpha(\hat{p}(wW_I, s)) = p(wW_I, s)$.
The paths from Figure \ref{pathchaser} are examples of paths created by this algorithm. Again, note that the algorithm is defined for more than just critical elements.

\begin{lemma}
 \label{lem:pathscancel}
Let $wW_I \in Perm_3(W)$, with maximal length element $w$, and consider $s \in I$, $s \leq m(wW_I)$. Then $\omega(\hat{p}(wW_I, s)) + \omega(p(wW_I, s)) = 0$, 
where $\omega$ is defined in Theorem \ref{thm:fundthm} part c. Moreover, these paths end at the same coset.
\end{lemma}

\begin{proof}
 Note that $Perm_3(W)$ is a cubical complex. In particular, given a coset $wW_I$ with maximum length element $w$, and $s \in S$, we see that the faces corresponding to 
$wsW_{I - s}$ and $wW_{I-s}$ are parallel faces of $wW_I$. In particular, one can show that $[wW_{I-s}: wW_I] = - [wsW_{I-s}: wW_I]$. 
We see that the products of incidence numbers appearing in the 
formula for $\omega(p(wW_I, s))$ are all $-1$, and the number of $-1$ terms is the same as the power of $-1$ appearing outside the product. Thus $p(wW_I, s) = [wW_{I-s}:wW_I]$, and similarly, $\hat{p}(wW_I, s) = [wsW_{I-s}:wW_I]$. 
However, these incident numbers are additive inverses, so their sum is $0$. The first result follows.

For the second result, consider the coset $wW_A$. Consider running the algorithm for $p(wW_I, s)$. At the $i$th step of the algorithm we consider $t_i = m(uW_{I -s})$ for some $u$, and append 
$ut_iW_{I - s_i}$ to the end of the path. Let $t_1, \ldots, t_r$ be the resulting sequence of reflections. Note that these reflections all come from $A_i$, and we see that $w t_1 \cdots t_r$ is an 
element of the last coset when the algorithm terminates. Observe that, since the algorithm terminates, $A \cap Des(w t_1 \cdots t_r) = \emptyset$. At each step, one can show that $w t_1 \cdots t_k$ is the maximum element of $w t_1 \cdots t_k W_J$. Since $A$ is a set of ascents for $w t_1 \cdots t_r$, $w t_1 \cdots t_r$ is the minimum length element of $wW_{A}$.
By similar arguments for $\hat{p}(wW_I, s)$, we obtain another sequence $t'_1, \ldots, t'_r$, such that the last coset of the path is $wt'_1\cdots t'_r w_J$, and $wt'_1\cdots t'_r$ is the minimum length element of $wW_A$. Thus, the paths $p(wW_I, s)$ and $\hat{p}(wW_I, s)$ end at the same coset $wt_1 \cdots t_r W_J$.
\end{proof}

We show that the following proposition is true, and we  
use it to define an involution. This proposition explains why we have been defining everything in this section for arbitrary cosets in $Perm_3(W)$, and not just critical ones.
\begin{proposition}
\label{prop:pathpattern}

Fix cosets $uW_I, vW_J \in Perm_3(W)$, with maximal coset representatives $u, v$, and let $P$ be an alternating, directed path from $uW_I$ to $vW_J$. Assume that $vW_J$ is 
critical, and that either $uW_I$ is critical, or $m(uW_I) \in I$. Then there exists an integer $m$,
paths $P_j$, cosets $w_jW_{I_j}$, and simple reflections $s_j$ for $j \in [m]$, paths $Q_j$ for $2 \leq j \leq m$, and a path $R$ such that:
\begin{enumerate}
 \item Either $p(w_jW_{I_j}, s_j) = P_j Q_j$ or $\hat{p}(w_jW_{I_j}, s_j) = P_j Q_j$ for $2 \leq j \leq m$,
 \item Either $p(w_1W_{I_1}, s_1) = P_1$ or $\hat{p}(w_1W_{I_1}, s_1) = P_1$,
 \item $m(w_jW_{I_j}) \in I_j$ for all $j$,
 \item $s_j \leq m(w_jW_{I_j})$ for all $j$,
 \item $P = P_m P_{m-1} \cdots P_1 R$,
 \item $w_mW_{I_m} = uW_I$.
\end{enumerate}

\end{proposition}

\begin{proof}
 We prove the result by induction on the length of $P$. Clearly $P$ has at least one edge. 
We claim that this edge must be of the form  $uW_I > uW_{I-s}$ or $uW_I > usW_{I-s}$ for some $s \in I$, $s \leq m(uW_I)$. 
Clearly this is the case if $m(uW_I) = \infty$, so suppose $uW_I$ is not critical, and the first edge is of the form $uW_I > uW_{I - s}$ for some $s > m(uW_I)$. Then we note that the matching algorithm matches $uW_{I - s}$ with $uW_{I-s-m}$ where $m = m(uW_I)$. However, this means that $P$ is not a directed alternating path, a contradiction.

Assume the first edge is of the form $uW_I uW_{I-s}$, where $s \in I$, and $s \leq m(uW_I)$. 
Then there exists $P', Q', R'$ such that $p(uW_I, s) = P' Q'$ and $P = P' R'$. Let 
$P'$ be the maximum of all such paths, let $wW_K$ be the last coset of $P'$. If $wW_K = vW_J$, then we are done.

Otherwise, observe that $M(wW_K) \subset wW_K$. That is, the last edge of $P'$ must be from the matching in order for $P$ to be an alternating directed path. In particular, the last edge must be of the form $M(wW_K), wW_K$. Therefore $m(wW_K) \in wW_K$.
By induction, there exists an integer $m$, simple reflections $s_j$, cosets $w_jW_{I_j}$, paths $P_j, Q_j$ and $R$ statisfying 1-6 for the 
path $R'$ from $wW_K$ to $vW_J$. Let $w_{m+1}W_{I_{m+1}} = uW_I$, $Q_{m+1} = Q', P_{m+1} = P'$, $s_{m+1} = s$. Clearly we have properties 1-6 for this collection.
A similar argument holds if the first edge of $P$ is of the form $uW_I > usW_{I-s}$ for some $s \in I$, $s < m(uW_I)$.
\end{proof}

Fix critical cells $uW_I$ and $vW_J$ with $|J| = |I| - 1$. Let $\mathscr{P}(uW_I, vW_J)$ denote the set of all alternating directed paths from $uW_I$ to $vW_J$. We wish to construct an involution 
on these paths. Let $P \in \mathscr{P}(uW_I, vW_J)$. Let $R, P_m, \ldots, P_1$ be paths which satisfy all the properties of Proposition \ref{prop:pathpattern} for $P$. 
Let $\alpha(P) = P_m \cdots, P_2, \alpha(P_1) R$. We claim that $\omega(\alpha(P)) + \omega(P) = 0$. Clearly $\omega(\alpha(P)) + \omega(P) = (\omega(\alpha(P_1)) + \omega(P)) \omega(R) \prod_{i=2}^m \omega(P_i) = 0$, 
since $\omega(\alpha(P_1)) + \omega(P_1) = 0$ by Lemma \ref{lem:pathscancel}. Also we see that $\alpha$ is an involution. Applying Theorem \ref{thm:fundthm} to $Perm_3(W)$, we get a complex 
$\Delta_c$ homotopy equivalent to $Perm_3(W)$. Moreover, as a result of our involution calculation, $[uW_I: vW_J] = 0$ in $\Delta_c$, and hence the boundary operator is the 
$0$-map. We can conclude:
\begin{theorem}
\label{optimalitykis3}
 The matching $M$ is an optimal matching for $Perm_3(W)$. 
\end{theorem}

\section{Conclusion and Open Problems}
We conclude with several open problems. First, it would be nice to understand the cohomology ring structure of $Perm_k(W)$, and the attachment maps of the minimcal cell complex we get via discrete Morse theory. It is already known how to use discrete Morse theory to study cup products, so there is hope in this direction. However, computing attachment maps is a very challenging problem.

It is also interesting to note that there is a natural group action of $W$ on $Perm_k(W)$, and this induces a group action on the cohomology groups. It would be of note of this group action is isomorphic to the group action on the cohomology of the complement. Moreover, could the representation be understood by acting on the (co)homology basis we have constructed? The first step here would be to understand our homology basis in terms of representative cycles in $Perm_k(W)$.

\bibliographystyle{amsplain}
\bibliography{k_parbib}

\pagebreak
\section{Appendix}

\begin{table}[htbp]
\begin{tabular}{|l|c|c|c|}
\hline
 Group & $k$ & $i$ & $\widetilde{\beta}_i(\mathcal{M}(\mathscr{W}_{n,k}))$ \\ \hline
$H_3$ & $3$ & $1$ & $31$ \\ \hline
$H_4$ & $3$ & $1$ & $3601$ \\ \hline
$H_4$ & $4$ & $2$ & $719$ \\ \hline
$F_4$ & $3$ & $1$ & $289$ \\ \hline
$F_4$ & $4$ & $2$ & $47$ \\ \hline
$E_6$ & $3$ & $1$ & $7201$ \\ \hline
$E_6$ & $3$ & $2$ & $720$ \\ \hline
$E_6$ & $4$ & $2$ & $5039$ \\ \hline
$E_6$ & $5$ & $3$ & $1441$ \\ \hline
$E_6$ & $6$ & $4$ & $125$ \\ \hline
$E_7$ & $3$ & $1$ & $135073$ \\ \hline
$E_7$ & $3$ & $2$ & $135072$ \\ \hline
$E_7$ & $4$ & $2$ & $141119$ \\ \hline
$E_7$ & $5$ & $3$ & $60481$ \\ \hline
$E_7$ & $6$ & $4$ & $11591$ \\ \hline
$E_7$ & $7$ & $5$ & $757$ \\ \hline
$E_8$ & $3$ & $1$ & $10946881$ \\ \hline
$E_8$ & $3$ & $2$ & $54492480$ \\ \hline
$E_8$ & $4$ & $2$ & $12337919$ \\ \hline
$E_8$ & $4$ & $4$ & $2177280$ \\ \hline
$E_8$ & $5$ & $3$ & $7257601$ \\ \hline
$E_8$ & $6$ & $4$ & $2600639$ \\ \hline
$E_8$ & $7$ & $5$ & $2600639$ \\ \hline
$E_8$ & $8$ & $6$ & $60481$ \\ \hline
\end{tabular}
\caption{nonzero Betti numbers of exceptional Coxeter groups}
 \label{exceptionalbettis}
\end{table}

\end{document}